\documentclass[oneside,english]{amsart}
\usepackage[latin9]{inputenc}
\usepackage{units}
\usepackage{mathrsfs}
\usepackage{amstext}
\usepackage{amsthm}
\usepackage{amssymb}
\usepackage{stmaryrd}

\makeatletter
\numberwithin{equation}{section}
\numberwithin{figure}{section}
\theoremstyle{plain}
\newtheorem{thm}{\protect\theoremname}[section]
\theoremstyle{definition}
\newtheorem{defn}[thm]{\protect\definitionname}
\theoremstyle{plain}
\newtheorem{prop}[thm]{\protect\propositionname}
\theoremstyle{remark}
\newtheorem{rem}[thm]{\protect\remarkname}
\theoremstyle{plain}
\newtheorem{cor}[thm]{\protect\corollaryname}
\theoremstyle{plain}
\newtheorem{lem}[thm]{\protect\lemmaname}
\theoremstyle{plain}
\newtheorem{assumption}[thm]{\protect\assumptionname}

\makeatother

\usepackage{babel}
\providecommand{\assumptionname}{Assumption}
\providecommand{\corollaryname}{Corollary}
\providecommand{\definitionname}{Definition}
\providecommand{\lemmaname}{Lemma}
\providecommand{\propositionname}{Proposition}
\providecommand{\remarkname}{Remark}
\providecommand{\theoremname}{Theorem}

\begin{document}
\title{Singular Set and Curvature Blow-Up Rate of The Level Set Flow}
\author{Siao-Hao Guo}
\begin{abstract}
\thanks{The research was partially supported by the grant 109-2115-M-002-018-MY3
of National Science and Technology Council of Taiwan. \medskip{}
}Under certain conditions such as the $2$-convexity, a singularity
of the level set flow is of type I (in the sense that the rate of
curvature blow-up is constrained before and after the singular time)
if and only if the flow shrinks to either a round point or a $C^{1}$
curve near that singular point. Analytically speaking, the arrival
time is $C^{2}$ near a critical point if and only if it satisfies
a \L ojasiewicz inequality near the point.
\end{abstract}

\maketitle
\tableofcontents{}

\section{\label{introduction}Introduction}

Given an initial hypersurface $\Sigma_{0}$ in $\mathbb{R}^{n}$ \footnote{Here $n\geq2$ is an integer.}that
is closed,\footnote{That is, compact without boundary.} connected,
smoothly embedded, and strictly mean-convex,\footnote{Namely, the sum of all principal curvatures, with respect to the inward
normal, is strictly positive. In fact, one can start with a mean-convex
hypersurface since the mean curvature flow will immediately become
strictly mean-convex by the strong maximum principle. } its evolution by the mean curvature vector, i.e., the mean curvature
flow, advances monotonically\footnote{Because the strict mean-convexity is preserved along the flow; specifically,
the minimum of the mean curvature of the time-slice is non-decreasing
by the maximum principle.} towards the inside\footnote{By the Jordan-Brouwer separation theorem, $\mathbb{R}^{n}\setminus\Sigma_{0}$
consists of two open connected sets, the inside $\Omega$ and the
outside $\mathbb{R}^{n}\setminus\bar{\Omega}$, where $\Omega$ is
called the region bounded by $\Sigma_{0}$.} until it becomes singular in finite time (cf. \cite{E}). Nevertheless,
there is a unique way to continue the flow through the singularities.

Let $\Omega$ be the domain bounded by $\Sigma_{0}$. Evans and Spruck
in \cite{ES} (see also \cite{CGG}) showed that there is a unique
viscosity solution $u\in C^{0,1}\left(\bar{\Omega}\right)$ to the
following degenerate elliptic equation
\begin{equation}
-\left(\mathrm{I}-\frac{\nabla u}{\left|\nabla u\right|}\varotimes\frac{\nabla u}{\left|\nabla u\right|}\right)\cdotp\nabla^{2}u=1\label{level set flow}
\end{equation}
using the elliptic regularization and the comparison principle for
viscosity solutions. The significance of the solution is that near
every regular point of $u$, that is, 
\begin{defn}
\label{regular point}A point $x\in\Omega$ is called a regular point
of $u$ provided that $u$ is smooth near $x$ with $\nabla u\left(x\right)\neq0$.
Otherwise, it is called a singular point.
\end{defn}

\noindent the equation (\ref{level set flow}) is satisfied pointwisely
and can be rewritten as 
\begin{equation}
-\nabla\cdot\frac{\nabla u}{\left|\nabla u\right|}=\frac{1}{\left|\nabla u\right|},\label{reformulated level set flow}
\end{equation}
meaning that the flow of the level hypersurfaces\footnote{Note that near a regular point, every level set of $u$ is a smoothly
embedded hypersurface. } is a mean curvature flow\footnote{The LHS of (\ref{reformulated level set flow}) is the mean curvature
of the level hypersurface and the RHS is the normal speed of the flow
of the level sets.} with positive mean curvature.\footnote{Because we have $H=\left|\nabla u\right|^{-1}$ by (\ref{reformulated level set flow}),
where $H$ denotes the mean curvature of the level hypersurface.} Moreover, it is proved in \cite{ES} that $\Sigma_{t}=\left\{ u=t\right\} $
prior to the first singular time, where $\Sigma_{t}$ is the mean
curvature flow starting at $\Sigma_{0}$. That is to say, the flow
of the level sets of $u$, which is called a \textbf{level set flow},\footnote{Our definition of a ``level set flow'' is different from (but closely
related to) the convention. For instance, see \cite{I1} to compare
the definitions.} is indeed an extension of the mean curvature flow starting at $\Sigma_{0}=\partial\Omega$.
For each point $x\in\Omega$, $u\left(x\right)$ is the time when
the flow travels through $x$, so $u$ is called the \textbf{arrival
time} of the level set flow. Note that the level set flow sweeps the
domain $\Omega$ and vanishes right after time $t=\max_{\Omega}u$,
which is called the extinction time.

Regarding the regularity of the level set flow, it is shown in \cite{I1}
(see also \cite{W1}) that $\mathcal{H}^{n-1}\lfloor\left\{ u=t\right\} $
defines a Brakke flow of integral varifolds. In particular, Huisken's
monotonicity formula (cf. \cite{H2} and \cite{I2}) holds for the
level set flow; it follows that the flow has finite entropy (cf. \cite{CM1}).
In addition, Ilmanen proved that for almost every $t\in\left(0,\max_{\Omega}u\right)$
the singular set on $u=t$ is of $\mathcal{H}^{n-1}$-measure zero.
Later, White in \cite{W1} showed that the Hausdorff dimension of
the singular set is at most $n-2$; furthermore, he proved in \cite{W2}
and \cite{W3} that the tangent flow (see \cite{I2}) at a singularity
is, up to a rigid motion, $\left\{ \sqrt{-t}\,\mathcal{C}_{k}\right\} _{t\in\left(-\infty,0\right)}$,
where
\begin{equation}
\mathcal{C}_{k}=S_{\sqrt{2\left(n-k-1\right)}}^{n-k-1}\times\mathbb{R}^{k},\label{generalized cylinder}
\end{equation}
for some $k\in\left\{ 0,\cdots,n-2\right\} $. As such, singular points
can be classified according to their associated tangent flows as follows:
\begin{defn}
A singular point is called a 
\[
\left\{ \begin{array}{c}
\textrm{round point,}\quad\textrm{if}\,\,k=0,\\
k\textrm{-cylindrical point,\quad\textrm{if}\,\,}k\in\left\{ 1,\cdots,n-2\right\} ,
\end{array}\right.
\]
where $k$ is the integer in (\ref{generalized cylinder}).\footnote{Whenever the cylindrical points are mentioned, we always assume that
$n\geq3$.} 
\end{defn}

Based on the noncollapsing property of the level set flow, Haslhofer
and Kleiner in \cite{HK} gave the smooth estimates (at regular points)
in terms of the mean curvature. Colding and Minicozzi in \cite{CM4}
established that $u\in C^{1,1}\left(\bar{\Omega}\right)$ and it is
twice differentiable everywhere on $\Omega$; moreover, they proved
that the singular points are exactly the critical points of $u$.
Regarding the further regularity, they proved in \cite{CM5} that
$u\in C^{2}\left(\bar{\Omega}\right)$ if and only if there is exactly
one singular time when the level set flow shrinks to either a round
point or a closed connected $C^{1}$ embedded $k$-manifold consisting
of $k$-cylindrical points, see Theorem \ref{globally C^2}. 

In this paper we give a localized version of Theorem \ref{globally C^2},
namely, 
\begin{thm}
\label{characterization of continuity of Hessian}The Hessian $\nabla^{2}u$
is continuous at a singular point if and only if the point is either
a round point or a $k$-cylindrical point near which the singular
set is a $C^{1}$ embedded $k$-manifold.
\end{thm}

\noindent To distinguish them from other singular points, such points
(as stated in Theorem \ref{characterization of continuity of Hessian})
are called \textbf{regular singular points} (see Definition \ref{regular singular point}).
As is shown in \cite{CM6}, if $p$ is a regular singular point, then
$u$ satisfies a \L ojasiewicz inequality
\[
\left|u\left(x\right)-u\left(p\right)\right|^{\frac{1}{2}}\leq\beta\left|\nabla u\left(x\right)\right|
\]
in a neighborhood of $p$ (see Proposition \ref{Lojasiewicz inequality}).
Since $\left|\nabla u\right|^{-1}$ is the mean curvature (see (\ref{reformulated level set flow})),
which is comparable with the second fundamental form\footnote{In other words, the largest principal curvature is comparable with
the normal speed of the flow.} by the pinching estimate in \cite{HK}, the \L ojasiewicz inequality
can be interpreted as the rate of blow-up of the second fundamental
form does not exceed the rate of $\left|t-u\left(p\right)\right|^{-\frac{1}{2}}$
near the singularity in space-time (see (\ref{classical type I singularity})).
Such a singularity is then called a \textbf{type I singularity} (see
Definition \ref{type I singularity}) of the level set flow. Unlike
the traditional definition of a type I singularity (of the mean curvature
flow) where the rate of blow-up of curvature is constrained only before
the (first) singular time, here the constraint is also imposed to
the flow past the singular time.\footnote{Because we require the \L ojasiewicz inequality to hold in a neighborhood
of $p$ in $\Omega$.} Note that in the case of a regular singular point, the flow locally
vanishes after the singular time; thus, the constraint is void after
the singular time. However, in the scenario of a neckpinch where the
flow is locally split into two pieces after the singular time,\footnote{That is, the singular point is a saddle point of $u$.}
the singularity is not of type I due to the rapid clearing-out phenomenon
in \cite{CM3} (see Theorem \ref{no type I saddle}), even though
the rate of blow-up of curvature is possibly bounded by the designated
rate prior to the singular time.

It turns out that the type I condition is not only necessary but also
sufficient for a singular point to be a regular singular point, so
long as the singular point is either a round point or a $1$-cylindrical
point (which is called a neckpinch singularity of the flow in \cite{SS}).
Specifically, we find the following:
\begin{thm}
\label{main theorem}Suppose that the singular set consists of only
round points and/or $1$-cylindrical points, then for any singular
point $p$ the following three statements are equivalent: 
\begin{enumerate}
\item $u$ is $C^{2}$ near $p$. 
\item The singular set near $p$ is either a round point or a $C^{1}$ embedded
curve.
\item The \L ojasiewicz inequality holds near $p$.
\end{enumerate}
\end{thm}

\noindent There are some situations where the hypothesis in Theorem
\ref{main theorem} is realized. For instance, 
\begin{prop}
\label{a priori restriction of singularity model}If either one of
the following conditions holds, only round points or $1$-cylindrical
points can arise as singular points:
\begin{itemize}
\item The dimension $n$ is two or three.
\item The initial hypersurface $\Sigma_{0}$ is strictly $2$-convex.
\item The entropy of $\Sigma_{0}$ is strictly lower than that of the $2$-cylinder
$\mathcal{C}_{2}$.
\end{itemize}
\end{prop}

\noindent In the second condition, the strictly $2$-convexity of
$\Sigma_{0}$ means that $\kappa_{1}+\kappa_{2}>0$, where $\kappa_{1}\leq\kappa_{2}\leq\cdots\leq\kappa_{n-1}$
denote the principal curvatures ordered by magnitude.\footnote{More generally, by $m$-convexity we mean that $\kappa_{1}+\cdots+\kappa_{m}\geq0$.
Note that $1$-convex is convex and $\left(n-1\right)$-convex is
mean-convex.} This is a condition stronger than the strict mean-convexity but weaker
than the strict convexity.\footnote{That is, $1$-convex (i.e. convex) $\Rightarrow$ $2$-convex $\Rightarrow\cdots\Rightarrow$
$\left(n-1\right)$-convex (i.e. mean-convex).} In addition, note that by the compactness of $\Sigma_{0}$ we have
\begin{equation}
\kappa_{1}+\kappa_{2}\geq\alpha H\label{uniformly 2-convex}
\end{equation}
for some $\alpha>0$. It is shown in \cite{CHN} (see also \cite{HK})
that the level set flow is uniformly $2$-convex in the sense that
(\ref{uniformly 2-convex}) holds at every regular point. In light
of the asymptotic behavior of the level set flow near a singular point
(see Remark \ref{Reifenberg}) and that $\mathcal{C}_{k}$ is not
$2$-convex if $k\geq2$, it is clear that the second condition in
Proposition \ref{a priori restriction of singularity model} yields
the result. As for the third condition, the result follows from the
observation that the entropy (see (\ref{entropy})) of the tangent
flow at any singular point is bounded above by $E\left[\Sigma_{0}\right]$
(cf. \cite{CM1}), where $E\left[\Sigma_{0}\right]$ is the entropy
of $\Sigma_{0}$, and that
\begin{equation}
1<E\left[\mathcal{C}_{0}\right]<E\left[\mathcal{C}_{1}\right]<E\left[\mathcal{C}_{2}\right]<\cdots<E\left[\mathcal{C}_{n-2}\right]<2\label{entropy level}
\end{equation}
(cf. \cite{CM1} and \cite{S}). 

Theorem \ref{characterization of continuity of Hessian} will be proved
in Section \ref{continuity of Hessian and singular set}. In Section
\ref{round points} we explain that $u$ is $C^{2}$ near any round
point. It is in Theorem \ref{continuity of Hessian} that we prove
Theorem \ref{characterization of continuity of Hessian} for the case
of cylindrical points. In addition, in Corollary \ref{locally C^2}
we show that the continuity of $\nabla^{2}u$ at a cylindrical point
$p$ yields that $u$ is $C^{2}$ near $p$. In the end of Section
\ref{continuity of Hessian and singular set} we will see how Colding-Minicozzi's
characterization of the globally $C^{2}$ regularity, i.e., Theorem
\ref{globally C^2}, is related to Theorem \ref{characterization of continuity of Hessian},
especially why the flow has only one singular time when $u\in C^{2}\left(\bar{\Omega}\right)$. 

For Theorem \ref{main theorem}, the equivalence of the first two
statements comes from the discussion in the last paragraph.\footnote{Note that if the singular set near a singular point $p$ is a $C^{1}$
embedded curve, then $p$ can never be a round point by Section \ref{round points}. } Thus, in Section \ref{Lojasiewicz inequality and type I singularity}
we focus on associating the regular singular points (see Definition
\ref{regular singular point}) with the type I singular points (see
Definition \ref{type I singularity}). Given that all regular points
must be of type I (see Proposition \ref{Lojasiewicz inequality}),
the last ingredient of proving Theorem \ref{main theorem} is to show
that a type I singular point, especially when it is a $1$-cylindrical
point,\footnote{By Section \ref{round points}, a round point is already a regular
singular point.} is a regular singular point. This is the bulk of Section \ref{neckpinch singularities},
where we are devoted to prove Theorem \ref{characterization of regular singular point}
with the help of Theorem \ref{no type I saddle} from Section \ref{saddle points}.

\section{\label{continuity of Hessian and singular set}Continuity of Hessian
and Singular Set}

The objective of this section is to prove Theorem \ref{characterization of continuity of Hessian},
the localized version of Theorem \ref{globally C^2} in \cite{CM5}.
To this end, we will first review Colding-Minicozzi's regularity theorem
in \cite{CM4} and the calculations therein. The analysis will be
divided into two cases according to the types of singularities: the
round points in Section \ref{round points} and the cylindrical points
in Section \ref{cylindrical points}. 

For convenience, from now on let us denote the singular set by $\mathcal{S}$,
the set of round points by $\mathcal{S}_{0}$, and the set of $k$-cylindrical
points by $\mathcal{S}_{k}$; thus we have
\[
\mathcal{S}=\mathcal{S}_{0}\cup\mathcal{S}_{1}\cup\cdots\cup\mathcal{S}_{n-2}.
\]
Also, as equation (\ref{level set flow}) is translation-invariant,\footnote{Namely, $x\mapsto u\left(x+p\right)-c$ satisfies (\ref{level set flow})
as well.} for ease of notation we hereafter assume that $0$ is a singular
point with $u\left(0\right)=0$.\footnote{Note that by this setting the point $0$ is an interior (singular)
point of the domain and the time $0$ is no longer the initial time
of the flow.} 

First of all, let us invoke the following theorem in \cite{CM4} on
the regularity of the level set flow and the characterization of singular
points.
\begin{thm}
\label{regularity}The solution $u$ to (\ref{level set flow}) is
globally $C^{1,1}$ and twice differentiable everywhere in the domain.
Furthermore, the singular points are exactly the critical points of
$u$. 
\end{thm}

\noindent Given a regular point $x$ in the domain, let $N=\frac{\nabla u}{\left|\nabla u\right|}$
be the inward unit normal vector and $\left\{ e_{1},\cdots,e_{n-1}\right\} $
be an orthonormal basis for the tangent hyperplane $T_{x}\left\{ u=u\left(x\right)\right\} $.
By the calculations in \cite{CM4} we have 
\[
-\nabla^{2}u\cdot\left(N\otimes N\right)=\left|\frac{A}{H}\right|^{2}+\frac{\triangle H}{H^{3}},
\]

\[
-\nabla^{2}u\cdot\left(N\otimes e_{i}\right)=\frac{\nabla_{i}H}{H^{2}},
\]
\[
-\nabla^{2}u\cdot\left(e_{i}\otimes e_{j}\right)=\frac{A_{ij}}{H},
\]
where $A$ denotes the second fundamental form (of the level hypersurface)
and $\triangle$ is the Laplace-Beltrami operator (on the level hypersurface).
It follows that the Hessian of $u$ can be written as
\begin{equation}
\nabla^{2}u\left(x\right)=-\left(N,e_{1},\cdots,e_{n-1}\right)\left(\begin{array}{cc}
\left|\frac{A}{H}\right|^{2}+\frac{\triangle H}{H^{3}} & \frac{\nabla H}{H^{2}}\\
\frac{\nabla H}{H^{2}} & \frac{A}{H}
\end{array}\right)\left(N,e_{1},\cdots,e_{n-1}\right)^{\textrm{T}},\label{Hessian at regular point}
\end{equation}
where $\left(N,e_{1},\cdots,e_{n-1}\right)$ denotes the orthogonal
$n\times n$ matrix consisting of the column vectors $\left\{ N,e_{1},\cdots,e_{n-1}\right\} $
and the notation T in the upper right corner stands for the transpose
of matrices. Note that the Hessian is invariant under the parabolic
reascaling.\footnote{The function $x\mapsto\lambda^{-2}u\left(\lambda x\right)$ is also
a solution to (\ref{level set flow}) and has the same Hessian as
$u$ at the corresponding points.}

\subsection{\label{round points}Round Points}

Let us assume in this subsection that $0$ is a round point. By Brakke's
regularity theorem (cf. \cite{B} and \cite{I1}), given $\epsilon>0$,
$M>\sqrt{2\left(n-1\right)}$, and $m\geq2$, there exist $t_{0}<0$
such that for every $t\in\left[t_{0},0\right)$, the set 
\[
\left(\frac{1}{\sqrt{-t}}\left\{ u=t\right\} \right)\cap B_{M}
\]
is a normal graph over the $\left(n-1\right)$-sphere $\mathcal{C}_{0}$
with $C^{m}$ norm at most $\epsilon$. After undoing the normalization
(with respect to the scale $\sqrt{-t}$) and weaving the level sets
all together, we see that the contour map of $u$ in $\left\{ u\geq t_{0}\right\} \cap B_{M\sqrt{t_{0}}}$
resembles the contour map of the function
\[
x\mapsto\frac{-1}{2\left(n-1\right)}\left|x\right|^{2},
\]
which is the trajectory of the shrinking sphere $\sqrt{-t}\,\mathcal{C}_{0}$
as $t\nearrow0$. By virtue of (\ref{Hessian at regular point}) and
the asymptotically spherical behavior of the level set flow, Colding
and Minicozzi showed in \cite{CM4} that 
\begin{equation}
\nabla^{2}u\left(0\right)=\frac{-1}{n-1}\,\mathrm{I}_{n}.\label{Hessian at round point}
\end{equation}
Therefore, $\nabla^{2}u$ is continuous at $0$ \footnote{Compare (\ref{Hessian at regular point}) with (\ref{Hessian at round point}).}and
$u$ is $C^{2}$ in a neighborhood of $0$.\footnote{The point $0$ is an isolated singularity.} 
\begin{rem}
\label{asymptotically spherical}On the other hand, if the level set
flow shrinks to a point $p$ \footnote{In this case, $p$ is a local maximum point of $u$, so by Theorem
\ref{regularity} it must be a singular point.}with the property for some $\delta>0$ the set $u=u\left(p\right)-\delta$
is a closed, connected, smoothly embedded, and convex hypersurface
in a neighborhood of $p$, then the level set flow in the region $\hat{\Omega}$
bounded by $u=u\left(p\right)-\delta$ near $p$ agrees with the mean
curvature flow starting at $\partial\hat{\Omega}$ (cf. \cite{ES}),
which by \cite{GH} (for $n=2$) and \cite{H1} (for $n\geq3$) shrinks
to a point (which must be $p$) in an asymptotically spherical manner.
Therefore, $p$ is a round point.
\end{rem}

\subsection{\label{cylindrical points}Cylindrical Points }

Throughout this subsection we assume that $0$ is a $k$-cylindrical
point for some $k\in\left\{ 1,\cdots,n-2\right\} $. Additionally,
by the rotational invariance of (\ref{level set flow})\footnote{Namely, $x\mapsto u\left(Q\,x\right)$ satisfies (\ref{level set flow})
for any orthogonal matrix $Q$.} we may assume that the tangent flow at $0$ is $\sqrt{-t}\,\mathcal{C}_{k}$.\footnote{That is to say, the orientation is chosen in such a way that the $k$-dimensional
axis of the tangent $k$-cylinder is $\left\{ 0\right\} \times\mathbb{R}^{k}$.} Also, we will adopt the following notation for coordinates:
\[
x=\left(y,z\right)\in\mathbb{R}^{n-k}\times\mathbb{R}^{k}.
\]
Let us begin the analysis of cylindrical points with the following
remarkable theorem in \cite{CM2} (see also \cite{CM3}). 
\begin{thm}
\label{asymptotically cylindrical}Given $\epsilon>0$ and $M>\sqrt{2\left(n-k-1\right)}$,
there exist $\delta>0$ and $L>0$ so that if for some $t_{0}<0$,
the set
\[
\left(\frac{1}{\sqrt{-2\,t_{0}}}\left\{ u=2\,t_{0}\right\} \right)\cap B_{L}
\]
is a normal graph over a cylinder congruent to $\mathcal{C}_{k}$
with $C^{m}$ norm at most $\delta$,\footnote{By Brakke's regularity theorem, such a $t_{0}$ does exist with the
base cylinder chosen to be $\mathcal{C}_{k}$ itself.} where $m\geq2$ is an absolute constant, then for every $t\in\left[t_{0},0\right)$,
the set
\[
\left(\frac{1}{\sqrt{-t}}\left\{ u=t\right\} \right)\cap B_{M}
\]
is a normal graph over $\mathcal{C}_{k}$ with $C^{m}$ norm no more
than $\epsilon$.\footnote{The crux is that the theorem provides a way to determine when does
the normalized flow starts to become close to $\mathcal{C}_{k}$ to
the specified degree. This is crucial to the work in \cite{CM3}.} 
\end{thm}

\noindent Recovering from the normalization (with respect to the
scale $\sqrt{-t}$) and having the level sets woven together, we see
that the contour map of $u$ in 
\[
\left\{ u\geq t_{0}\right\} \cap B_{M\sqrt{t_{0}}}\setminus\mathscr{C}_{\phi},
\]
\footnote{In contrast to the circumstance near a round point (see Section \ref{round points}),
here the information about $u$ is vague inside the cone $\mathscr{C}_{\phi}$,
making it more intricate to analyze. }where
\begin{equation}
\mathscr{C}_{\phi}=\left\{ \,\left|y\right|\leq\left|z\right|\tan\phi\right\} \label{cone}
\end{equation}
is the solid cone with apex $0$, axis $\left\{ 0\right\} \times\mathbb{R}^{k}$,
and angle 
\begin{equation}
\phi=\sin^{-1}\frac{\sqrt{2\left(n-k-1\right)}}{M},\label{angle}
\end{equation}
approximates to the contour map of the function 
\begin{equation}
\left(y,z\right)\mapsto\frac{-1}{2\left(n-k-1\right)}\left|y\right|^{2},\label{cylindrical arrival time}
\end{equation}
which is the trajectory of the shrinking truncated $k$-cylinder $\sqrt{-t}\,\left(\mathcal{C}_{k}\cap B_{M}\right)$
as $t\nearrow0$. This motivates the following definition. 
\begin{defn}
\label{cylindrical scale}(Cylindrical Scale) Given positive constants
$\phi$ and $\epsilon$, the level set flow is said to be $\left(\phi,\epsilon\right)$-asymptotically
cylindrical near the $k$-cylindrical point $0$ on the cylindrical
scale $r$ provided that for every $t<0$, in $B_{r}\setminus\mathscr{C}_{\phi}$
the two sets $u=t$ and $\left|y\right|^{2}+2\left(n-k-1\right)t=0$
are $\epsilon$-close in the $C^{m}$ topology after rescaling by
the factor $\frac{1}{\sqrt{-t}}$.\footnote{By Theorem \ref{asymptotically cylindrical}, $r=\sqrt{2\left(n-k-1\right)\left(-t_{0}\right)}$
is a qualified cylindrical scale, where the constants $\phi$ and
$M$ are related by (\ref{angle}). } 
\end{defn}

\noindent Regarding the two parameters $\phi$ and $\epsilon$ in
Definition \ref{cylindrical scale}, $\phi$ can be used to determine
the extent\footnote{By (\ref{angle}), the smaller the angle $\phi$, the larger the radius
$M$.} of the asymptotic regime after the normalization and $\epsilon$
if a measure of the degree to which the normalized level set of $u$
is close to $\mathcal{C}_{k}$. 
\begin{rem}
\label{Reifenberg}By rechoosing $t_{0}$ in Theorem \ref{asymptotically cylindrical}
even closer to $0$, we may assume (owing to Brakke's regularity theorem)
that 
\[
\left(\frac{1}{\sqrt{-2\,t_{0}}}\left\{ u=2\,t_{0}\right\} \right)\cap B_{2L}
\]
is a normal graph over $\mathcal{C}_{k}$ with $C^{m}$ norm at most
$\frac{\delta}{2}$. Then applying Theorem \ref{asymptotically cylindrical}
to the nearby $k$-cylindrical points, if any, yields the following.
There exists $\rho>0$ (depending on $n$, $k$, $m$, $t_{0}$, and
the Lipschitz constant of $u$) so that every $k$-cylindrical point
in $B_{\rho}^{n-k}\times B_{\rho}^{k}$ has a common $\left(\phi,\epsilon\right)$-asymptotically
cylindrical scale $r>0$ (depending on $n$, $k$, and $t_{0}$, see
the footnote in Definition \ref{cylindrical scale}).

Furthermore, when $\phi$ and $\epsilon$ are chosen sufficiently
small (depending on $n$ and $k$), by \cite{CM3} the set $\mathcal{S}_{k}\cap\left(B_{\rho}^{n-k}\times B_{\rho}^{k}\right)$
is contained in the graph $y=\psi\left(z\right)$, where 
\[
\psi:B_{\rho}^{k}\rightarrow B_{\rho}^{n-k}
\]
is a map with small Lipschitz norm.\footnote{The Lipschitz map $\psi$ comes from the McShane-Whitney extension
of $\mathcal{S}_{k}\cap\left(B_{\rho}^{n-k}\times B_{\rho}^{k}\right)$.} Moreover, in case that $\mathcal{S}_{k}\cap\left(B_{\rho}^{n-k}\times B_{\rho}^{k}\right)$
agrees with the graph of $\psi$, then 
\[
\mathcal{S}\cap\left(B_{\rho}^{n-k}\times B_{\rho}^{k}\right)=\mathcal{S}_{k}\cap\left(B_{\rho}^{n-k}\times B_{\rho}^{k}\right)
\]
\footnote{In other words, singular points in $B_{\rho}^{n-k}\times B_{\rho}^{k}$
are all $k$-cylindrical points.}and it turns out that $\psi\in C^{1}\left(B_{\rho}^{k}\right)$; the
tangent space of $\mathcal{S}\cap\left(B_{\rho}^{n-k}\times B_{\rho}^{k}\right)$
at each point is exactly the $k$-dimensional axis of the tangent
$k$-cylinder at that point.
\end{rem}

\begin{rem}
\label{no local minimum}In view of the asymptotic behavior of the
level set flow near a singular point (spherical or cylindrical), a
critical point of $u$ must be either a local maximum point or a saddle
point. That is to say, $u$ has no interior local minimum points.\footnote{A more fundamental way to see this is perhaps by the definition of
viscosity supersolution to equation (\ref{level set flow}) in  \cite{ES}.}
\end{rem}

Due to (\ref{entropy level}) and the upper semi-continuity of the
Gaussian density (cf. \cite{E}), in a neighborhood of the $k$-cylindrical
point $0$, the degree of every singular point is no higher than $k$.
The following theorem from \cite{CM3}, which is based on the aforementioned
property and Remark \ref{Reifenberg}, describes the structure of
the singular set near $0$.
\begin{thm}
\label{rectifiability} There exists $\rho>0$ so that 
\[
\mathcal{S}\cap\left(B_{\rho}^{n-k}\times B_{\rho}^{k}\right)=\left(\mathcal{S}_{0}\cup\cdots\cup\mathcal{S}_{k}\right)\cap\left(B_{\rho}^{n-k}\times B_{\rho}^{k}\right).
\]
Moreover, we have 
\begin{enumerate}
\item The set $\mathcal{S}_{k}\cap\left(B_{\rho}^{n-k}\times B_{\rho}^{k}\right)$
is contained in a Lipschitz graph $y=\psi\left(z\right)$, where $\psi:B_{\rho}^{k}\rightarrow B_{\rho}^{n-k}$
is a map with small Lipschitz norm. 
\item The Hausdorff dimension of $\left(\mathcal{S}_{0}\cup\cdots\cup\mathcal{S}_{k-1}\right)\cap\left(B_{\rho}^{n-k}\times B_{\rho}^{k}\right)$
is at most $k-1$.
\end{enumerate}
\end{thm}

Analogous to (\ref{Hessian at round point}) for round points, Colding
and Minicozzi in \cite{CM4} showed that the Hessian at the $k$-cylindrical
point $0$ is
\begin{equation}
\nabla^{2}u\left(0\right)=\left(\begin{array}{cc}
\frac{-1}{n-k-1}\,\mathrm{I}_{n-k}\\
 & \mathrm{O}_{k}
\end{array}\right).\label{Hessian at cylindrical point}
\end{equation}

\begin{rem}
\label{discriminant}The Hessian at a different $k$-cylindrical point,
say $p$, is similar to the matrix (\ref{Hessian at cylindrical point})
in the sense that 
\[
\nabla^{2}u\left(p\right)=Q^{-1}\,\nabla^{2}u\left(0\right)\,Q
\]
for some orthogonal matrix $Q$ with $\textrm{Ker}\,\nabla^{2}u\left(p\right)$
the $k$-dimensional axis of the tangent $k$-cylinder at $p$. Because
of the rigid forms of Hessian at singular points (see (\ref{Hessian at round point})
and (\ref{Hessian at cylindrical point})), one can tell the type
of a singular point according to the Hessian (such as from the trace)
at that point. 
\end{rem}

Regarding the continuity of Hessian near $0$, we have the following
preliminary remark based on the asymptotically cylindrical behavior
of the level set flow outside the cone $\mathscr{C}_{\phi}$ (cf.
\cite{CM4}). 
\begin{rem}
\label{continuity of Hessian outside the cone}Given positive constants
$\phi$ and $\epsilon$, there there exists $r>0$ so that the level
set flow is $\left(\phi,\epsilon\right)$-asymptotically cylindrical
in $B_{r}$. Thus, for any point in $B_{r}\setminus\mathscr{C}_{\phi}$,
we can choose the tangent vectors in (\ref{Hessian at regular point})
in such a way that $\left\{ e_{1},\cdots,e_{n-k-1}\right\} $ is close
to the round-direction\footnote{That is, parallel to $S_{\sqrt{2\left(n-k-1\right)}}^{n-k-1}\times\left\{ 0\right\} .$}
of $\mathcal{C}_{k}$ and $\left\{ e_{n-k},\cdots,e_{n-1}\right\} $
is close to the axial-direction\footnote{That is, parallel to $\left\{ 0\right\} \times\mathbb{R}^{k}$.}
of $\mathcal{C}_{k}$. By comparing (\ref{Hessian at regular point})
with (\ref{Hessian at cylindrical point}) and taking Definition \ref{cylindrical scale}
into account, $\nabla^{2}u$ can be made as close to $\nabla^{2}u\left(0\right)$
as pleased in $B_{r}\setminus\mathscr{C}_{\phi}$, so long as $\epsilon$
is chosen sufficiently small.
\end{rem}

Unlike the situation in Section \ref{round points} where the Hessian
is automatically continuous at a round point, a priori we only know
that $\nabla^{2}u\left(x\right)\rightarrow\nabla^{2}u\left(0\right)$
as $x\rightarrow0$ from outside of the cone $\mathscr{C}_{\phi}$.
In fact, the continuity of Hessian from the inside of the cone depends
on the regularity of the singular set $\mathcal{S}$ near $0$, as
will be seen in the following theorem. The proof is based on the ideas
in \cite{CM5}. 
\begin{thm}
\label{continuity of Hessian}$\nabla^{2}u$ is continuous at the
$k$-cylindrical point $0$ if and only if near $0$ the singular
set $\mathcal{S}$ is a $C^{1}$ embedded $k$-manifold. 

In this case, $0$ is a local maximum point of $u$ near which we
have $\mathcal{S}=\left\{ u=0\right\} $.\footnote{Geometrically speaking, near the point $0$ the level set flow shrinks
to $\mathcal{S}$ at time $0$ and then vanishes.} Moreover, every singular point $x$ near $0$ must be a $k$-cylindrical
point with $T_{x}\mathcal{S}$ the $k$-dimensional axis of the tangent
$k$-cylinder at $x$.\footnote{By Remark \ref{discriminant}, this is $\textrm{Ker}\,\nabla^{2}u\left(x\right)$.}
\end{thm}

\begin{proof}
($\Rightarrow$) Let $\phi$ and $\epsilon$ be small positive constants
(as is required in Remark \ref{Reifenberg}), there exists $\delta>0$
such that for every $k$-cylindrical point $x=\left(y,z\right)$ in
$\bar{B}_{\delta}^{n-k}\times\bar{B}_{\delta}^{k}$ the level set
flow is $\left(\phi,\epsilon\right)$-asymptotically cylindrical about
$x$ with cylindrical scale $\delta$; moreover,
\begin{equation}
\left|\nabla^{2}u\left(x\right)-\nabla^{2}u\left(0\right)\right|<\label{continuity of Hessian: Hessian}
\end{equation}
\[
\frac{1}{\sqrt{n}}\,\min\left\{ \frac{1}{n-k-1},\,\,\min_{0\leq l\leq n-2,\,\,l\neq k}\left|\frac{1}{n-k-1}-\frac{1}{n-l-1}\right|\right\} .
\]
for every $x=\left(y,z\right)$ in $\bar{B}_{\delta}^{n-k}\times\bar{B}_{\delta}^{k}$. 

Fix $z_{0}\in B_{\delta}^{k}$. Let $y_{0}\in\bar{B}_{\delta}^{n-k}\cap\left\{ z=z_{0}\right\} $
be such that 
\begin{equation}
u\left(y_{0},z_{0}\right)=\max_{\left|y\right|\leq\delta}\,u\left(y,z_{0}\right).\label{continuity of Hessian: maximum point on each level}
\end{equation}
Note that a priori such a point $y_{0}$ might not be unique. In view
of the normal vector field $N=\frac{\nabla u}{\left|\nabla u\right|}$
on $\left\{ \left|y\right|\geq\left|z_{0}\right|\tan\phi\right\} \cap\left\{ z=z_{0}\right\} $
(with the asymptotically cylindrical behavior of the level set flow
taken into account), we infer that $\left|y_{0}\right|<\left|z_{0}\right|\tan\phi$.
Then it follows from (\ref{continuity of Hessian: maximum point on each level})
that 
\begin{equation}
\nabla u\left(y_{0},z_{0}\right)\in\left\{ 0\right\} \times\mathbb{R}^{k}.\label{continuity of Hessian: gradient}
\end{equation}
We claim that $\nabla u\left(y_{0},z_{0}\right)=0$, i.e., $x_{0}=\left(y_{0},z_{0}\right)$
is a critical point. To prove the claim, suppose the contrary that
$\nabla u\left(x_{0}\right)\neq0$, so by Theorem \ref{regularity}
$x_{0}$ would be a regular point. By (\ref{Hessian at cylindrical point}),
(\ref{continuity of Hessian: Hessian}), and (\ref{continuity of Hessian: gradient})
we would have 
\[
-\left(\mathrm{I}-\frac{\nabla u}{\left|\nabla u\right|}\left(x_{0}\right)\varotimes\frac{\nabla u}{\left|\nabla u\right|}\left(x_{0}\right)\right)\cdotp\nabla^{2}u\left(0\right)=\frac{n-k}{n-k-1},
\]
\[
-\left(\mathrm{I}-\frac{\nabla u}{\left|\nabla u\right|}\left(x_{0}\right)\varotimes\frac{\nabla u}{\left|\nabla u\right|}\left(x_{0}\right)\right)\cdotp\left(\nabla^{2}u\left(x_{0}\right)-\nabla^{2}u\left(0\right)\right)>\frac{-1}{n-k-1},
\]
Thus,
\[
-\left(\mathrm{I}-\frac{\nabla u}{\left|\nabla u\right|}\left(x_{0}\right)\varotimes\frac{\nabla u}{\left|\nabla u\right|}\left(x_{0}\right)\right)\cdotp\nabla^{2}u\left(x_{0}\right)>1
\]
contradicting (\ref{level set flow}).

Following from the last paragraph, given $z_{0}\in B_{\delta}^{k}$
there exists $y_{0}\in\mathscr{C}_{\phi}\cap\left\{ z=z_{0}\right\} $
such that $x_{0}=\left(y_{0},z_{0}\right)\in\mathcal{S}$. Using Remark
\ref{discriminant} and (\ref{continuity of Hessian}) we deduce that
$x_{0}$ is indeed a $k$-cylindrical point. It then follows from
Remark \ref{Reifenberg} that 
\[
\mathcal{S}\cap\left(B_{\delta}^{n-k}\times B_{\delta}^{k}\right)=\mathcal{S}_{k}\cap\left(B_{\delta}^{n-k}\times B_{\delta}^{k}\right)
\]
is a graph $y=\psi\left(z\right)$, where $\psi:B_{\delta}^{k}\rightarrow B_{\delta}^{n-k}$
is a $C^{1}$ map with $\psi\left(0\right)=0$, and that for every
$x\in\mathcal{S}\cap\left(B_{\delta}^{n-k}\times B_{\delta}^{k}\right)$,
$T_{x}\mathcal{S}$ is the axis of the tangent cylinder of the level
set flow at $x$. Note in particular that the maximum point in (\ref{continuity of Hessian: maximum point on each level})
turns out to be unique and that $y_{0}=\psi\left(z_{0}\right)$. On
the other hand, by Theorem \ref{regularity} we have
\[
\frac{\partial}{\partial z}\left[u\left(\psi\left(z\right),z\right)\right]=\nabla u\left(\psi\left(z\right),z\right)\cdot\left[\frac{\partial\psi\left(z\right)}{\partial z},\mathrm{I}_{k}\right]=0,
\]
which, by the mean value theorem, implies that 
\[
u\left(y_{0},z_{0}\right)=u\left(\psi\left(z_{0}\right),z_{0}\right)=u\left(0\right).
\]
Therefore, $u\left(0\right)$ is the maximum value of $u$ in $B_{\delta}^{n-k}\times B_{\delta}^{k}$
with
\[
\left\{ u=0\right\} \cap\left(B_{\delta}^{n-k}\times B_{\delta}^{k}\right)=\mathcal{S}\cap\left(B_{\delta}^{n-k}\times B_{\delta}^{k}\right).
\]

$\left(\Leftarrow\right)$ Firstly, note that by Theorem \ref{rectifiability},
the $k$-cylindrical points must be densely distributed in the $C^{1}$
embedded $k$-manifold $\mathcal{S}$. It then follows from the upper
semicontinuity of the Gaussian density (cf. \cite{E}) that $\mathcal{S}=\mathcal{S}_{k}$
near $0$.

Given $\eta>0$, choose $\delta>0$ such that $\mathcal{S}\cap\left(B_{\delta}^{n-k}\times B_{\delta}^{k}\right)=\mathcal{S}_{k}\cap\left(B_{\delta}^{n-k}\times B_{\delta}^{k}\right)$
is a graph $y=\psi\left(z\right)$, where $\psi:B_{\delta}^{k}\rightarrow B_{\delta}^{n-k}$
is $C^{1}$ with small gradient; in particular, in view of Remark
\ref{Reifenberg} and Remark \ref{discriminant}, we may assume that
\begin{equation}
\left|\nabla^{2}u\left(\psi\left(z\right),z\right)-\nabla^{2}u\left(0\right)\right|\leq\frac{\eta}{2}\label{continuity of Hessian: vertical continuity}
\end{equation}
for every $z\in B_{\delta}^{k}$. Additionally, by Remark \ref{Reifenberg}
and Remark \ref{continuity of Hessian outside the cone} we may further
assume that every $k$-cylindrical point in $B_{\delta}^{n-k}\times B_{\delta}^{k}$
has the uniform cylindrical scale $\delta$ with the parameters $\phi$
and $\epsilon$ (see Definition \ref{cylindrical scale}) sufficiently
small (depending on $n$, $k$, and $\eta$) that
\begin{equation}
\left|\nabla^{2}u\left(y,z\right)-\nabla^{2}u\left(\psi\left(z\right),z\right)\right|\leq\frac{\eta}{2}\label{continuity of Hessian: horizontal continuity}
\end{equation}
for every $y\in B_{\delta}^{n-k}$ and $z\in B_{\delta}^{k}$. Combining
(\ref{continuity of Hessian: vertical continuity}) and (\ref{continuity of Hessian: horizontal continuity})
yield 

\[
\left|\nabla^{2}u\left(x\right)-\nabla^{2}u\left(0\right)\right|\leq\eta
\]
for every $x\in B_{\delta}^{n-k}\times B_{\delta}^{k}$. Thus, $\nabla^{2}u$
is continuous at $0$.
\end{proof}
If $\nabla^{2}u$ is continuous at the $k$-cylindrical point $0$,
then by (the ``only if'' part of) Theorem \ref{continuity of Hessian},
$\mathcal{S}=\mathcal{S}_{k}$ near $0$ and it is a $C^{1}$ embedded
$k$-manifold. On the other hand, applying (the ``if'' part of)
Theorem \ref{continuity of Hessian} to the nearby singular points
yields that $\nabla^{2}u$ is continuous at every singular point near
$0$ and hence it is continuous on a neighborhood of $0$. Thus, we
have the following corollary.
\begin{cor}
\label{locally C^2}The Hessian $\nabla^{2}u$ is continuous at a
$k$-cylindrical point $0$ if and only if $u\in C^{2}$ in a neighborhood
of $0$.
\end{cor}

Let us conclude this section with the following proof of the characterization
of globally $C^{2}$ regularity in \cite{CM5} using Theorem \ref{continuity of Hessian}.
The focus is on seeing why the flow has only one singular time when
$u$ is globally $C^{2}$.
\begin{thm}
\label{globally C^2}The solution $u$ to (\ref{level set flow})
is globally $C^{2}$ if and only if the first singular time $T$ of
the level set flow is the time of extinction\footnote{Namely, $T=\max u$. In other words, the mean curvature flow stating
at the initial hypersurface becomes singular at time $T$ and then
vanishes.} and that the set $u=T$ is either a single round point or a closed
connected $C^{1}$ embedded $k$-manifold consisting of $k$-cylindrical
points for some $k\in\left\{ 1,\cdots,n-2\right\} $.
\end{thm}

\begin{proof}
$\left(\Leftarrow\right)$ It is clear from Theorem \ref{regularity},
Section \ref{round points}, and Theorem \ref{continuity of Hessian}.

$\left(\Rightarrow\right)$ Let $T$ be the first singular time and
choose a singular point $p\in\left\{ u=T\right\} $. If $p$ is a
round point, then by Section \ref{round points} there exists $\delta>0$
so that in a neighborhood of $p$ the set $u=t$ is asymptotic to
a sphere for every $t\in\left(T-\delta,T\right)$. For every $t\in\left(T-\delta,T\right)$,
the asymptotically spherical part of $u=t$ near $p$ is clearly compact
and relatively open in $\left\{ u=t\right\} $; therefore, it is indeed
the the whole of $u=t$ by the connectedness of $\left\{ u=t\right\} $.\footnote{Note that for every $t\in\left(0,T\right)$, the hypersurface $u=t$
is diffeomorphic to $u=0$ and hence is connected.} In light of this, the flow shrinks to the point $p$ at time $T$
and so the theorem is proved. Thus, for the rest of our discussion,
let us assume that $p$ is a $k$-cylindrical point for some $k\in\left\{ 1,\cdots,n-2\right\} $.

Let $\Gamma$ be the connected component of $p$ in the set of all
$k$-cylindrical points in $u=T$. By Theorem \ref{continuity of Hessian}
$\Gamma$ is a closed\footnote{In proving the compactness of $\Gamma$, we use the continuity of
$u$ and the upper semicontinuity of Gaussian density as well.} connected $C^{1}$ embedded $k$-manifold. Let $\phi$ and $\epsilon$
be sufficiently small positive constants,\footnote{This is for the application of Remark \ref{Reifenberg}. In addition,
for the latter purpose we require that $\epsilon<2-\sqrt{2}$.} by Remark \ref{Reifenberg} and the compactness of $\Gamma$, there
exists $\delta>0$ so that the $\left(\phi,\epsilon\right)$-asymptotically
cylindrical scale (see Definition \ref{cylindrical scale}) of any
point on $\Gamma$ is at least $\delta$. It follows that for every
$q\in\Gamma$, in $B_{\delta}\left(q\right)\setminus\mathscr{C}_{\phi}^{\Gamma}\left(q\right)$
the set $u=t$, whenever it is nonempty, is $\epsilon\sqrt{T-t}$-close
in the Hausdorff sense to the $k$-cylinder $\mathcal{C}_{k}^{\Gamma,\sqrt{T-t}}\left(q\right)$,
where $\mathscr{C}_{\phi}^{\Gamma}\left(q\right)$ is the cone obtained
by rotating the cone $\mathscr{C}_{\phi}$ (see (\ref{cone})) to
have the axis $T_{q}\Gamma$ and then translating by the vector $q$
(so that $q$ is the apex of $\mathscr{C}_{\phi}^{\Gamma}\left(q\right)$);
$\mathcal{C}_{k}^{\Gamma,\sqrt{T-t}}\left(q\right)$ is the $k$-cylinder
obtained by firstly rotating the $k$-cylinder $\mathcal{C}_{k}$
(see (\ref{generalized cylinder})) to have the axis $T_{q}\Gamma$,
then scaling by the factor $\sqrt{T-t}$ (so that the radius becomes
$\sqrt{2\left(n-k-1\right)\left(T-t\right)}$), and lastly translating
by the vector $q$. In particular, the distance from $q$ to any point
on $\left\{ u=t\right\} \cap B_{\delta}\left(q\right)\cap\Pi_{n-k}^{\Gamma}\left(q\right)$,
where $\Pi_{n-k}^{\Gamma}\left(q\right)$ is the $\left(n-k\right)$-dimensional
plane passing through $q$ and orthogonal to $T_{q}\Gamma$, is no
more than $\left(\sqrt{2\left(n-k-1\right)}+\epsilon\right)\sqrt{T-t}$.

As such, for every $t\in\left[0,T\right)$ with $T-t\leq\frac{\delta^{2}}{32\left(n-k-1\right)}$,
the set
\[
\hat{\Sigma}_{t}=\left\{ x:u\left(x\right)=t\,\,\textrm{and}\,\,\textrm{dist}\left(x,\Gamma\right)\leq\frac{\delta}{2}\right\} .
\]
is nonempty and compact; moreover, we will show that $\hat{\Sigma}_{t}$
is relatively open in $u=t$, so by the connectedness\footnote{Note that the hypersurface $u=t$ is diffeomorphic to the hypersurface
$u=0$ when $t<T$.} of $u=t$ we then have $\hat{\Sigma}_{t}=\left\{ u=t\right\} $.
To prove the relatively openness of $\hat{\Sigma}_{t}$ in $u=t$,
let us fix $x\in\hat{\Sigma}_{t}$. By the compactness of $\Gamma$
there exists $q\in\Gamma$ so that $\textrm{dist}\left(x,\Gamma\right)=\left|x-q\right|\leq\frac{\delta}{2}$;
moreover, $x-q$ is orthogonal $T_{q}\Gamma$ and hence $x\in\Pi_{n-k}^{\Gamma}\left(q\right)$.
Since $x\in\left\{ u=t\right\} \cap B_{\delta}\left(q\right)\cap\Pi_{n-k}^{\Gamma}\left(q\right)$,
it follows from the discussion in the last paragraph that 
\[
\left|x-q\right|<\left(\sqrt{2\left(n-k-1\right)}+\epsilon\right)\sqrt{T-t}<\frac{\delta}{2\sqrt{2}},
\]
where the last inequality is due to $T-t\leq\frac{\delta^{2}}{32\left(n-k-1\right)}$
and $\epsilon<2-\sqrt{2}$. Then for any $r\in\left(0,\frac{\sqrt{2}-1}{2\sqrt{2}}\delta\right)$
we have $B_{r}\left(x\right)\cap\left\{ u=t\right\} \subset\hat{\Sigma}_{t}$;
namely, $x$ has an open neighborhood in $\left\{ u=t\right\} $ that
is contained in $\hat{\Sigma}_{t}$.

To finish the proof, it suffices to show that $\Gamma=\left\{ u=T\right\} $.
Were $\Gamma\neq\left\{ u=T\right\} $, then by Theorem \ref{continuity of Hessian}
and the compactness of both $\left\{ u=T\right\} $ and $\Gamma$,
we would have 
\begin{equation}
\textrm{dist}\left(\left\{ u=T\right\} \setminus\Gamma,\,\Gamma\right)=\varrho>0.\label{globally C^2: distance}
\end{equation}
On the other hand, from the result\footnote{That is, $\hat{\Sigma}_{t}=\left\{ u=t\right\} $.}
and argument in the last paragraph, we infer that the hypersurface
$u=t$ is indeed contained in the $\left(\sqrt{2\left(n-k-1\right)}+\epsilon\right)\sqrt{T-t}$-tubular
neighborhood of $\Gamma$ for every $t\in\left[0,T\right)$ with $T-t\leq\frac{\delta^{2}}{32\left(n-k-1\right)}$.
Particularly, for every $t\in\left[0,T\right)$ sufficiently close
to $T$ the set $u=t$ is contained in the $\frac{\varrho}{2}$-tubular
neighborhood of $\Gamma$, contradicting (\ref{globally C^2: distance})
because the level set flow is non-fattening (cf. \cite{I1} and \cite{W1}).
\end{proof}

\section{\label{Lojasiewicz inequality and type I singularity}\L ojasiewicz
Inequality and Type I Singularity}

In this section we aim to prove Theorem \ref{main theorem}, which,
in view of Theorem \ref{characterization of continuity of Hessian}
and the footnote in Definition \ref{regular singular point}, amounts
to characterizing the regular singular points (see Definition \ref{regular singular point})
by means of the type I condition (see Definition \ref{type I singularity})
under the hypothesis that $\mathcal{S}=\mathcal{S}_{0}\cup\mathcal{S}_{1}$.
This is done in three steps. Firstly, it is seen in Proposition \ref{Lojasiewicz inequality}
that the type I condition is a necessary condition for a singular
point to be a regular singular point. Next, in Section \ref{saddle points}
we prove that a type I singular point must be a local maximum point
(see Theorem \ref{no type I saddle}). Lastly, the proof is completed
in Section \ref{neckpinch singularities} by establishing Theorem
\ref{characterization of regular singular point}.

To start with, let us introduce the following definition. 
\begin{defn}
\label{regular singular point}(Regular Singular Points) A singular
point of $u$ is called a regular singular point if near which $u$
is $C^{2}$.\footnote{A round point is automatically a regular singular point (see Section
\ref{round points}). For a $k$-cylindrical point $p$, $p$ is a
regular singular point $\Leftrightarrow$ $\nabla^{2}u$ is continuous
at $p$ (see Corollary \ref{locally C^2}) $\Leftrightarrow$ there
exists $\delta>0$ so that $\mathcal{S}\cap B_{\delta}\left(p\right)$
is a $C^{1}$ embedded $k$-manifold (see Theorem \ref{continuity of Hessian}).}
\end{defn}

\noindent As pointed out in \cite{CM6}, near a regular singular
point the solution $u$ must satisfy a \L ojasiewicz inequality. For
the sake of completeness, we write the complete statement in the following
proposition and provide a proof.
\begin{prop}
\label{Lojasiewicz inequality}If $0$ is a regular singular point,
then for every $\epsilon>0$ there exists $\delta>0$ so that
\[
\left|u\right|^{\frac{1}{2}}\leq C_{n,k,\epsilon}\left|\nabla u\right|
\]
in $B_{\delta}$, where $k=\textrm{nullity\ensuremath{\left[\ensuremath{\nabla^{2}u}\left(0\right)\right]\,\in\left\{  0,\cdots,n-2\right\} } }$
and $C_{n,k,\epsilon}\rightarrow\sqrt{\frac{n-k-1}{2}}$ as $\epsilon\searrow0$.
\end{prop}

\begin{proof}
Let us assume that $k>0$ so $0$ is a $k$-cylindrical point; the
proof for $k=0$ (i.e., round point) is similar. 

Given $\epsilon>0$ (sufficiently small depending on $n$ and $k$),
choose $\delta>0$ such that $u$ is $C^{2}$ in $B_{\delta}$ and
that the singular set 
\begin{equation}
\mathcal{S}\cap B_{\delta}=\mathcal{S}_{k}\cap B_{\delta}=\left\{ u=0\right\} \cap B_{\delta}\label{Lojasiewicz inequality: singularity}
\end{equation}
is a $C^{1}$ embedded $k$-manifold (see Theorem \ref{continuity of Hessian})
whose normal bundle $N\mathcal{S}$ parametrizes $B_{\delta}$. Particularly,
for every $x\in B_{\delta}$ there exists a unique $x'\in\mathcal{S}\cap B_{\delta}$
such that $x-x'\in N_{x'}\mathcal{S}$.\footnote{Since $\left(N_{x'}\mathcal{S}\right)^{\perp}=T_{x'}\mathcal{S}$
is the axis of the tangent cylinder of the level set flow at $x'$
(cf. Theorem \ref{continuity of Hessian}), $x$ is on the $n-k$
dimensional plane passing through $x'$ and orthogonal to the axis
of the tangent cylinder at $x'$.} Furthermore, by the continuity of Hessian and Remark \ref{Reifenberg},
we may also assume that 
\begin{equation}
\left|\nabla^{2}u\left(x\right)-\nabla^{2}u\left(x'\right)\right|\leq\epsilon,\label{Lojasiewicz inequality: Hessian}
\end{equation}
\begin{equation}
\left|\frac{\nabla u\left(x\right)}{\left|\nabla u\left(x\right)\right|}+\frac{x-x'}{\left|x-x'\right|}\right|\leq\epsilon\label{Lojasiewicz inequality: normal}
\end{equation}
in case when $x\in B_{\delta}\setminus\mathcal{S}.$

To prove the proposition, let us fix $x\in B_{\delta}$ and assume
that $x\notin\mathcal{S}$; otherwise $u\left(x\right)=0$ and the
\L ojasiewicz inequality would hold trivially. Let $x'\in\mathcal{S}\cap B_{\delta}$
be as defined in the last paragraph. Since $u\left(x'\right)=0$ by
(\ref{Lojasiewicz inequality: singularity}) and $\nabla u\left(x'\right)=0$
(cf. Theorem \ref{regularity}), the Taylor's theorem says that there
exist $x_{0}$ on the segment $\overline{x'x}$ such that
\[
\left|u\left(x\right)\right|=\left|u\left(x\right)-u\left(x'\right)\right|=\left|\frac{1}{2}\,\nabla^{2}u\left(x_{0}\right)\cdot\left[\left(x-x'\right)\otimes\left(x-x'\right)\right]\right|
\]
\begin{equation}
\leq\frac{1}{2}\left(\frac{1}{n-k-1}+\epsilon\right)\left|x-x'\right|^{2}\label{Lojasiewicz inequality: function}
\end{equation}
Note that the last inequality is due to (\ref{Lojasiewicz inequality: Hessian}).
Likewise, by the mean vaule theorem, (\ref{Lojasiewicz inequality: Hessian}),
(\ref{Lojasiewicz inequality: normal}), and (\ref{Hessian at cylindrical point})
(see also Remark \ref{discriminant}), we get 
\[
\left|\nabla u\left(x\right)\right|=\nabla u\left(x\right)\cdot\,\frac{\nabla u\left(x\right)}{\left|\nabla u\left(x\right)\right|}=\left(\nabla u\left(x\right)-\nabla u\left(x'\right)\right)\cdot\,\frac{\nabla u\left(x\right)}{\left|\nabla u\left(x\right)\right|}
\]
\[
=\nabla^{2}u\left(x_{1}\right)\cdot\left[\left(x-x'\right)\otimes\frac{\nabla u\left(x\right)}{\left|\nabla u\left(x\right)\right|}\right],
\]
where $x_{1}\in\overline{xx'}$ is a point arising from the mean value
theorem,
\[
\geq\nabla^{2}u\left(x'\right)\cdot\left[\left(x-x'\right)\otimes\frac{\nabla u\left(x\right)}{\left|\nabla u\left(x\right)\right|}\right]-\epsilon\left|x-x'\right|
\]
\[
\geq\nabla^{2}u\left(x'\right)\cdot\left[\left(x-x'\right)\otimes-\frac{x-x'}{\left|x-x'\right|}\right]-\frac{\epsilon}{n-k-1}\left|x-x'\right|-\epsilon\left|x-x'\right|
\]
\begin{equation}
=\left(\frac{1}{n-k-1}-\epsilon\frac{n-k}{n-k-1}\right)\left|x-x'\right|,\label{Lojasiewicz inequality: gradient}
\end{equation}
Note that in the last equality we use the fact that $x-x'$ is orthogonal
to $T_{x'}\mathcal{S}=\textrm{Ker}\,\nabla^{2}u\left(x'\right)$ (cf.
Theorem \ref{continuity of Hessian}). Lastly, combining (\ref{Lojasiewicz inequality: function})
with (\ref{Lojasiewicz inequality: gradient}) yields 
\[
\left|u\left(x\right)\right|^{\frac{1}{2}}\leq C_{n,k,\epsilon}\left|\nabla u\left(x\right)\right|,
\]
where $C_{n,k,\epsilon}=\nicefrac{\sqrt{\frac{1}{2}\left(\frac{1}{n-k-1}+\epsilon\right)}}{\left(\frac{1}{n-k-1}-\epsilon\frac{n-k}{n-k-1}\right)}$.
\end{proof}
On the set of regular points, the level set flow is a strictly mean-convex
mean curvature flow with $H=\left|\nabla u\right|^{-1}$ (see (\ref{reformulated level set flow})).
It can be seen that the mean curvature increases without bound as
tending to critical points of $u$, which by Theorem \ref{regularity}
are singular points. In the case where the point $0$ is a regular
singular point, the level set flow is a mean curvature flow near the
point $0$ prior to time $0$; then it shrinks to a lower dimensional
set $\mathcal{S}$ near the point $0$ at time $0$ and vanishes afterwards
(see Theorem \ref{characterization of continuity of Hessian}). So
the solution $u$ in Proposition \ref{Lojasiewicz inequality} is
actually non-positive near the point $0$ and the \L ojasiewicz inequality
can be reformulated as
\begin{equation}
\sqrt{-t}\leq C_{n,k,\epsilon}H^{-1}\,\Leftrightarrow\,H\leq\frac{C_{n,k,\epsilon}}{\sqrt{-t}}\label{classical type I singularity}
\end{equation}
as $t\nearrow0$. Taking into account the fact that the mean curvature
is comparable with the norm of the second fundamental form (cf. \cite{HK}),
the singular point $0$ is the so-called ``type I singularity''
of the flow at the singular time $0$ (cf. \cite{H2}). 

As is seen in the last paragraph, the information of type I singularities
is incorporated in the \L ojasiewicz inequality; moreover, compared
with (\ref{classical type I singularity}), the \L ojasiewicz inequality
holds even at singularities of the flow. Thus, it seems reasonable
to define the notion of type I singularities of the level set flow
using the \L ojasiewicz inequality. This is the following definition. 
\begin{defn}
\label{type I singularity}The singular point $0$ is called a type
I singular point if there exists $\beta>0$ so that the \L ojasiewicz
inequality
\[
\left|u\right|^{\frac{1}{2}}\leq\beta\left|\nabla u\right|
\]
holds in a neighborhood of $0$.\footnote{The corresponding definition for a different singular point, say $p$,
should be modified as $\left|u-u\left(p\right)\right|^{\frac{1}{2}}\leq C\left|\nabla u\right|$
in some neighborhood of $p$. } Otherwise, it is called a type II singularity.
\end{defn}

\begin{rem}
\label{regular value}If $0$ is a type I singular point, then by
Theorem \ref{regularity} the singular set $\mathcal{S}$ near $0$
is contained in $\left\{ u=0\right\} $ (cf. \cite{CM6}). In other
words, in a neighborhood of $0$, a point $x$ is a regular point
provided $u\left(x\right)\neq0$. 
\end{rem}

\subsection{\label{saddle points}Saddle Points }

The highlight of this subsection is Theorem \ref{no type I saddle},
which says that all saddle points are type II singular points (see
Definition \ref{type I singularity}). As there are no interior local
minimum points (see Remark \ref{no local minimum}), a type I singular
point must be a local maximum point. This is essential to establishing
Theorem \ref{characterization of regular singular point}, which is
the last piece of puzzle of proving Theorem \ref{main theorem}. 

The proof of Theorem \ref{no type I saddle} is based on Proposition
\ref{saddle point} and Proposition \ref{clearing-out}; the latter
is the so-called rapid clearing-out phenomenon of the level set flow
(cf. \cite{CM3}). In this subsection we assume that $0$ is a saddle
point.
\begin{prop}
\label{saddle point}Suppose that the saddle point $0$ is a type
I singular point. Then there exist $\delta>0$ and an integral curve
\[
x\left(s\right)\in C_{loc}^{1}\left(0,\delta\right]\cap C\left[0,\delta\right]
\]
of the normal vector field $N=\frac{\nabla u}{\left|\nabla u\right|}$
that starts at $0$ at time $0$; that is, $x\left(0\right)=0$ and
\[
\frac{dx}{ds}=\frac{\nabla u\left(x\right)}{\left|\nabla u\left(x\right)\right|}\quad\textrm{for}\:\,0<s\leq\delta.
\]
Moreover, for every $s\in\left[0,\delta\right]$ we have $u\left(x\left(s\right)\right)>0$
and 
\[
\left|x\left(s\right)\right|\leq2\beta\sqrt{u\left(x\left(s\right)\right)},
\]
where $\beta$ is the positive constant in Definition \ref{type I singularity}.
\end{prop}

\begin{proof}
Since the level set flow and the \L ojasiewicz inequality are invariant
under the parabolic scaling,\footnote{Namely, $x\mapsto\lambda^{-2}u\left(\lambda x\right)$ satisfies (\ref{level set flow})
and the \L ojasiewicz inequality for every $\lambda>0$. } for simplicity let us assume that the \L ojasiewicz inequality in
Definition \ref{type I singularity} holds in $B_{1}$. 

As $0$ is a saddle point of $u$, there exists a sequence of points
$\left\{ p_{i}\right\} _{i\in\mathbb{N}}$ in $B_{1}$ that converges
to $0$ with $u\left(p_{i}\right)>u\left(0\right)=0$ for every $i$.
Because $\nabla u$ is Lipschitz continuous (cf. Theorem \ref{regularity}),
the existence and uniqueness theorem of ODE gives that for each $i$
there is a unique flow line of the vector field $\nabla u$ passing
through the point $p_{i}$.\footnote{Note that $p_{i}$ is not a stationary point of $\nabla u$ by the
\L ojasiewicz inequality.} Moreover, these flow lines are indeed smooth curves due to Theorem
\ref{regularity} \footnote{Each of these flow lines consists of regular points as $\nabla u$
vanishes at no points along the curve.} and the regularity theory in ODE. It then follows from a reparametrization
by the arc length that for every $i$ there is a smooth curve $x_{i}\left(s\right)$
in $B_{1}$ satisfying 
\begin{equation}
\frac{d}{ds}x_{i}=N\left(x_{i}\right)=\frac{\nabla u\left(x_{i}\right)}{\left|\nabla u\left(x_{i}\right)\right|}\label{saddle point: integral curve}
\end{equation}
with $x_{i}\left(0\right)=p_{i}$. It is clear that $x_{i}$ is defined
for $0\leq s\leq s_{i}$ with $s_{i}\rightarrow1$ as $i\rightarrow\infty$.

By the chain rule we have
\[
\frac{d}{ds}\left[u\left(x_{i}\right)\right]=\nabla u\left(x_{i}\right)\cdot N\left(x_{i}\right)=\left|\nabla u\left(x_{i}\right)\right|,
\]
which particularly implies that $u\left(x_{i}\left(s\right)\right)>0$
for  $0\leq s\leq s_{i}$; moreover, invoking the \L ojasiewicz inequality
gives 
\[
\frac{d}{ds}\left[u\left(x_{i}\right)\right]=\left|\nabla u\left(x_{i}\right)\right|\geq\frac{1}{\beta}\sqrt{u\left(x_{i}\right)},
\]
that is, $\frac{d}{ds}\sqrt{u\left(x_{i}\right)}\geq\frac{1}{2\beta}$.
Thus, we obtain
\begin{equation}
\beta\left|\nabla u\left(x_{i}\left(s\right)\right)\right|\geq\sqrt{u\left(x_{i}\left(s\right)\right)}\geq\sqrt{u\left(x_{i}\left(s\right)\right)}-\sqrt{u\left(x_{i}\left(0\right)\right)}\geq\frac{s}{2\beta}\label{saddle point: arc length}
\end{equation}
for $0\leq s\leq s_{i}$. In addition, since
\[
\frac{d^{2}}{ds^{2}}x_{i}=\frac{d}{ds}\left[\frac{\nabla u\left(x_{i}\right)}{\left|\nabla u\left(x_{i}\right)\right|}\right]=\frac{\left(\nabla^{2}u\right)N-\left(\nabla^{2}u\cdot\left(N\otimes N\right)\right)N}{\left|\nabla u\right|}\left(x_{i}\right),
\]
by (\ref{saddle point: arc length}) we have\footnote{The Hessian of $u$ is bounded according Theorem \ref{regularity}.}
\begin{equation}
\left|\frac{d^{2}}{ds^{2}}x_{i}\right|\leq\frac{2\left\Vert \nabla^{2}u\right\Vert _{L^{\infty}\left(B_{1}\right)}}{\left|\nabla u\left(x_{i}\right)\right|}\leq\frac{4\beta^{2}\left\Vert \nabla^{2}u\right\Vert _{L^{\infty}\left(B_{1}\right)}}{s}\label{saddle point: derivative estimate}
\end{equation}
for $0\leq s\leq s_{i}$.

In view of (\ref{saddle point: integral curve}) and (\ref{saddle point: derivative estimate}),
the Arzelà-Ascoli compactness theorem implies that $\left\{ x_{i}\left(s\right)\right\} _{i\in\mathbb{N}}$
subconverges to $x\left(s\right)$ in $C\left[0,1-\epsilon\right]\cap C^{1}\left[\epsilon,1-\epsilon\right]$
for every $\epsilon\in\left(0,1\right)$. Passing (\ref{saddle point: arc length})
to the limit (and noting that $\nabla u$ is continuous) gives
\begin{equation}
\beta\left|\nabla u\left(x\left(s\right)\right)\right|\geq\sqrt{u\left(x\left(s\right)\right)}\geq\frac{s}{2\beta}>0\label{saddle point: inequality}
\end{equation}
for $0<s\leq1$. Likewise, with the help of (\ref{saddle point: inequality}),
taking the limit of (\ref{saddle point: integral curve}) gives
\[
\frac{dx}{ds}=\frac{\nabla u\left(x\right)}{\left|\nabla u\left(x\right)\right|}
\]
for $0<s\leq1$ with $x\left(0\right)=0$. Lastly, noting that
\[
\left|x\left(s\right)\right|\leq\int_{0}^{s}\left|x'\left(\xi\right)\right|d\xi=s,
\]
(\ref{saddle point: inequality}) implies
\[
\sqrt{u\left(x\left(s\right)\right)}\geq\frac{s}{2\beta}\geq\frac{\left|x\left(s\right)\right|}{2\beta}
\]
for $0\leq s\leq1$.
\end{proof}
Proposition \ref{clearing-out} is the rapid clearing-out lemma in
\cite{CM3}. For reader's convenience, a proof will be provided. To
streamline the proof, we need Lemma \ref{small Gaussian area} concerning
the Gaussian area. Recall that given $p\in\mathbb{R}^{n}$ and $\Lambda>0$,
the Gaussian area with center $p$ and scale $\sqrt{\Lambda}$ of
a $\left(n-1\right)$-rectifiable set $\Sigma$ in $\mathbb{R}^{n}$
is defined as 
\begin{equation}
F_{p,\Lambda}\left(\Sigma\right)=\int_{\Sigma}\frac{e^{-\frac{\left|x-p\right|^{2}}{4\Lambda}}}{\left(4\pi\Lambda\right)^{\frac{n-1}{2}}}\,d\mathcal{H}^{n-1}\left(x\right).\label{Gaussian area}
\end{equation}
The entropy of $\Sigma$ (cf. \cite{CM1}) is defined as 
\begin{equation}
E\left[\Sigma\right]=\sup\left\{ F_{p,\Lambda}\left(\Sigma\right):p\in\mathbb{R}^{n},\Lambda>0\right\} .\label{entropy}
\end{equation}

\begin{lem}
\label{small Gaussian area}Given $M>0$, $k\in\left\{ 1,\cdots,n-2\right\} $,
and $\lambda\geq1$, there exist $\Lambda>2$ (depending on $n,k$)
and $\phi,\epsilon>0$ (depending on $n,k,M,\lambda$) with the following
property: If $\Sigma$ is a $\left(n-1\right)$-rectifiable set in
$\mathbb{R}^{n}$ with entropy no higher than $\lambda$; additionally,
$\Sigma$ is $\epsilon$-close in the $C^{1}$ topology to the $k$-cylinder
$\mathcal{C}_{k}$ in the ball $B_{\sqrt{2\left(n-k-1\right)}\csc\phi}$.
Then 
\[
F_{p,\Lambda}\left(\Sigma\right)\leq\frac{1}{2}
\]
for every $p\in B_{M\sqrt{\Lambda-1}}$.
\end{lem}

\begin{proof}
Firstly, we claim that there exists $\Lambda>2$ depending on $n$
and $k$ such that 
\begin{equation}
F_{p,\Lambda}\left(\mathcal{C}_{k}\right)\leq\frac{1}{8}\label{small Gaussian area: cylinder}
\end{equation}
for every $p\in\mathbb{R}^{n}$. To see that, let us adopt the notation
$x=\left(y,z\right)\in\mathbb{R}^{n-k}\times\mathbb{R}^{k}$ for the
following calculation. Given $p=\left(p_{1},p_{2}\right)\in\mathbb{R}^{n-k}\times\mathbb{R}^{k}$
and $\Lambda>1$, Tonelli's theorem gives 
\[
F_{p,\Lambda}\left(\mathcal{C}_{k}\right)=\left(\int_{S_{\sqrt{2\left(n-k-1\right)}}^{n-k-1}}\frac{e^{-\frac{\left|y-p_{1}\right|^{2}}{4\Lambda}}}{\left(4\pi\Lambda\right)^{\frac{n-k-1}{2}}}\,d\mathcal{H}^{n-k-1}\left(y\right)\right)\left(\int_{\mathbb{R}^{k}}\frac{e^{-\frac{\left|z-p_{2}\right|^{2}}{4\Lambda}}}{\left(4\pi\Lambda\right)^{\frac{k}{2}}}\,d\mathcal{H}^{k}\left(x\right)\right)
\]
\[
\leq\frac{1}{\left(4\pi\Lambda\right)^{\frac{n-k-1}{2}}}\,\mathcal{H}^{n-k-1}\left(S_{\sqrt{2\left(n-k-1\right)}}^{n-k-1}\right),
\]
from which the claim can be verified easily. This is how the constant
$\Lambda=\Lambda\left(n,k\right)$ is determined. Below we show how
to choose the constants $\phi$ and $\epsilon$.

Let $\Sigma$ be as stated and fix $p\in B_{M\sqrt{\Lambda-1}}$.
Using the technique in \cite{CIM} we have 
\[
F_{p,\Lambda}\left(\Sigma\setminus B_{\sqrt{2\left(n-k-1\right)}\csc\phi}\right)=\int_{\Sigma\setminus B_{\sqrt{2\left(n-k-1\right)}\csc\phi}}\left(\sqrt{2}^{n-1}e^{-\frac{\left|x-p\right|^{2}}{8\Lambda}}\right)\frac{e^{-\frac{\left|x-p\right|^{2}}{8\Lambda}}}{\left(8\pi\Lambda\right)^{\frac{n-1}{2}}}\,d\mathcal{H}^{n-1}\left(x\right)
\]
\[
\leq\sqrt{2}^{n-1}e^{-\frac{\left(\sqrt{2\left(n-k-1\right)}\csc\phi-M\sqrt{\Lambda-1}\right)^{2}}{8\Lambda}}F_{p,2\Lambda}\left(\Sigma\right)\leq\sqrt{2}^{n-1}e^{-\frac{\left(\sqrt{2\left(n-k-1\right)}\csc\phi-M\sqrt{\Lambda-1}\right)^{2}}{8\Lambda}}\lambda.
\]
Thus, by choosing $\phi=\phi\left(n,k,M,\lambda\right)$ sufficiently
small we have
\begin{equation}
F_{p,\Lambda}\left(\Sigma\setminus B_{\sqrt{2\left(n-k-1\right)}\csc\phi}\right)\leq\frac{1}{4}.\label{small Gaussian area: outside}
\end{equation}
To proceed, let us parametrize $\Sigma$ and $\mathcal{C}_{k}$ in
$B_{\sqrt{2\left(n-k-1\right)}\csc\phi}$ as
\[
x_{\Sigma}:\mathcal{M}\rightarrow B_{\sqrt{2\left(n-k-1\right)}\csc\phi}\quad\textrm{and}\quad x_{\mathcal{C}_{k}}:\mathcal{M}\rightarrow B_{\sqrt{2\left(n-k-1\right)}\csc\phi}
\]
respectively, where $\mathcal{M}$ is some compact $\left(n-1\right)$-manifold
diffeomorphic to $\mathcal{C}_{k}\cap B_{\sqrt{2\left(n-k-1\right)}\csc\phi}$,
in such a way that the two immersions are $\epsilon$-close in the
$C^{1}$ norm. It follows that
\[
F_{p,\Lambda}\left(\Sigma\cap B_{\sqrt{2\left(n-k-1\right)}\csc\phi}\right)-F_{p,\Lambda}\left(\mathcal{C}_{k}\cap B_{\sqrt{2\left(n-k-1\right)}\csc\phi}\right)
\]
\[
=\frac{1}{\left(4\pi\Lambda\right)^{\frac{n-1}{2}}}\int_{\mathcal{M}}\left(e^{-\frac{\left|x_{\Sigma}-p\right|^{2}}{4\Lambda}}\sqrt{\deg\left(g_{\Sigma}\right)}-e^{-\frac{\left|x_{\mathcal{C}_{k}}-p\right|^{2}}{4\Lambda}}\sqrt{\deg\left(g_{\mathcal{C}_{k}}\right)}\right)d\mathcal{H}^{n-1},
\]
where $g_{\Sigma}$ and $g_{\mathcal{C}_{k}}$ are the respective
induced metric on $\mathcal{M}$. Consequently, if $\epsilon=\epsilon\left(n,k,M,\lambda\right)$
is sufficiently small we have
\[
\left|F_{p,\Lambda}\left(\Sigma\cap B_{\sqrt{2\left(n-k-1\right)}\csc\phi}\right)-F_{p,\Lambda}\left(\mathcal{C}_{k}\cap B_{\sqrt{2\left(n-k-1\right)}\csc\phi}\right)\right|\leq\frac{1}{8},
\]
which together with (\ref{small Gaussian area: cylinder}) yield
\begin{equation}
F_{p,\Lambda}\left(\Sigma\cap B_{\sqrt{2\left(n-k-1\right)}\csc\phi}\right)\leq\frac{1}{4}.\label{small Gaussian area: inside}
\end{equation}
Lastly, the conclusion follows from (\ref{small Gaussian area: outside})
and (\ref{small Gaussian area: inside}).
\end{proof}
\begin{prop}
\label{clearing-out}Given that $0$ is a saddle point, for any $M>0$
there exists $\tau>0$ so that 
\[
\left\{ u=t\right\} \cap B_{M\sqrt{t}}=\emptyset
\]
for every $t\in\left(0,\tau\right]$. 
\end{prop}

\begin{proof}
As $0$ is a saddle point, it must be a $k$-cylindrical point for
some $k\in\left\{ 1,\cdots,n-2\right\} $ (see Section \ref{round points}).
Now set $\lambda$ as the entropy of the level set flow\footnote{More accurately, $\lambda$ is chosen to be the entropy of the initial
time-slice; the entropy of the latter time-slice does not exceed $\lambda$
owing to Huisken's monotonicity formula (cf. \cite{I2} and \cite{CM1}).} and let $\Lambda$, $\phi$, and $\epsilon$ be the constants from
Lemma \ref{small Gaussian area}. By Remark \ref{Reifenberg}, there
exists $\tau>0$ such that $u$ is $\left(\phi,\epsilon\right)$-asymptotically
cylindrical near the $k$-cylindrical point $0$ on the cylindrical
scale $\sqrt{-\tau}$; particularly, $\frac{1}{\sqrt{-t}}\Sigma_{t}$
is $\epsilon$-close in the $C^{1}$ topology to the $k$-cylinder
$\mathcal{C}_{k}$ in the ball \textbf{$B_{\sqrt{2\left(n-k-1\right)}\csc\phi}$}
for every $t\in\left[-\tau,0\right)$,\footnote{See Definition \ref{cylindrical scale}.}
where $\Sigma_{t}=\left\{ u=t\right\} $.

Given $t\in\left[-\tau,0\right)$ and $p\in B_{M\sqrt{\left(\Lambda-1\right)\left(-t\right)}}$,
the change of variable in Gaussian integrals (see (\ref{Gaussian area}))
and Lemma \ref{small Gaussian area} yield
\[
F_{p,\Lambda\left(-t\right)}\left(\Sigma_{t}\right)=F_{\frac{p}{\sqrt{-t}},\Lambda}\left(\frac{1}{\sqrt{-t}}\Sigma_{t}\right)\leq\frac{1}{2}.
\]
Then it follows from Huisken's monotonicity formula (cf. \cite{H2})
that the Gaussian density of the flow at the point $p$ and time $\left(\Lambda-1\right)\left(-t\right)$
satisfies
\[
\Theta\left[p,\left(\Lambda-1\right)\left(-t\right)\right]=\lim_{s\nearrow\Lambda\left(-t\right)}F_{p,\Lambda\left(-t\right)-s}\left(\Sigma_{t+s}\right)\leq F_{p,\Lambda\left(-t\right)}\left(\Sigma_{t}\right)\leq\frac{1}{2}.
\]
So we must have $u\left(p\right)\neq\left(\Lambda-1\right)\left(-t\right)$,
i.e., 
\[
p\notin\Sigma_{\left(\Lambda-1\right)\left(-t\right)}.
\]
To see this, suppose the contrary that $u\left(p\right)=\left(\Lambda-1\right)\left(-t\right)$.
As the singular set has dimension at most $n-2$, we can choose a
sequence of regular points $p_{i}\rightarrow p$, and we automatically
obtain $u\left(p_{i}\right)\rightarrow u\left(p\right)$ by the continuity
of $u$. The upper semicontinuity of Gaussian density yields 
\[
\Theta\left[p,\left(\Lambda-1\right)\left(-t\right)\right]=\Theta\left[p,u\left(p\right)\right]\geq\limsup_{i\rightarrow\infty}\,\Theta\left[p_{i},u\left(p_{i}\right)\right]\geq1
\]
(cf. \cite{E}), which is a contradiction. 

Finally, we get
\[
\Sigma_{\left(\Lambda-1\right)\left(-t\right)}\cap B_{M\sqrt{\left(\Lambda-1\right)\left(-t\right)}}=\emptyset\quad\forall\,\,t\in\left[-\tau,0\right),
\]
which, by setting $t'=\left(\Lambda-1\right)\left(-t\right)$, can
be rewritten as
\[
\Sigma_{t'}\cap B_{M\sqrt{t'}}=\emptyset\quad\forall\,\,t'\in\left(0,\left(\Lambda-1\right)\tau\right].
\]
\end{proof}
We are now in the position to prove the main theorem of this subsection. 
\begin{thm}
\label{no type I saddle}The saddle point $0$ is a type II singular
point. 

In other words, a type I point cannot be a saddle point and hence
(by Remark \ref{no local minimum}) must be a local maximum point. 
\end{thm}

\begin{proof}
By Proposition \ref{clearing-out} with $M=3\beta$ , where $\beta$
is the positive constant in Definition \ref{type I singularity},
there exists $\tau>0$ such that 
\[
\Sigma_{t}\cap B_{3\beta\sqrt{t}}=\emptyset
\]
for every $t\in\left(0,\tau\right]$, where $\Sigma_{t}=\left\{ u=t\right\} .$
On the other hand, by Proposition \ref{saddle point} there exists
a continuous curve $x\left(s\right)$ satisfying $u\left(x\left(0\right)\right)=u\left(0\right)=0$,
$u\left(x\left(s\right)\right)>0$, and
\[
x\left(s\right)\in\Sigma_{u\left(x\left(s\right)\right)}\cap\bar{B}_{2\beta\sqrt{u\left(x\left(s\right)\right)}}
\]
for every $s\in\left[0,\delta\right]$. A contradiction then follows. 
\end{proof}
\begin{rem}
\label{characterization of regular point}Remark \ref{regular value}
can be improved by Theorem \ref{no type I saddle} as follows: When
$0$ is a type I singular point, there exists a neighborhood of $0$
in which a point $x$ is a regular point if and only if $u\left(x\right)\neq0$
($\Leftrightarrow u\left(x\right)<0)$. 
\end{rem}

\subsection{\label{neckpinch singularities}Neckpinch Singularity}

In this subsection $0$ is assumed to be a $1$-cylindrical point.
Our goal is to prove 
\begin{thm}
\label{characterization of regular singular point}If the $1$-cylindrical
point $0$ is a type I singular pont (see Definition \ref{type I singularity}),
then it must be a regular singular point (see Definition \ref{regular singular point}).
\end{thm}

\noindent by the method of contradiction, so let us assume throughout
this subsection that
\begin{assumption}
\label{assumption of characterization theorem}The $1$-cylindrical
point $0$ is a type I singular point but fails to be a regular singular
point. 
\end{assumption}

\noindent  and we manage to get a contradiction. 

\medskip{}

For ease of notations, the orientation is chosen in such a way that
the tangent flow at $0$ is the self-shrinking of the $1$-cylinder
\[
\mathcal{C}_{1}=S_{\sqrt{2\left(n-2\right)}}^{n-2}\times\mathbb{R}^{1}
\]
and we will adopt the notation $x=\left(y,z\right)\in\mathbb{R}^{n-1}\times\mathbb{R}$
for coordinates. 

Let $\phi,\epsilon\in\left(0,1\right)$ be small constants to be determined.\footnote{The constant $\phi$ will be determined in (\ref{choice of phi}),
where $\beta$ is the constant in Definition \ref{type I singularity},
and then $\epsilon$ will be chosen sufficiently small depending on
$n$ and the choice of $\phi$. } By Remark \ref{Reifenberg} and the parabolic rescaling,\footnote{If $u$ satisfies (\ref{level set flow}) and is $\left(\phi,\epsilon\right)$-asymptotically
cylindrical in $B_{r}$, then $x\mapsto r^{-2}u\left(rx\right)$ satisfies
(\ref{level set flow}) and is $\left(\phi,\epsilon\right)$-asymptotically
cylindrical in $B_{1}$.} we may assume that the level set flow is $\left(\phi,\epsilon\right)$-asymptotically
cylindrical in $B_{1}^{n-1}\times\left(-1,1\right)$. Likewise, by
the type I condition and the parabolic rescaling,\footnote{If $u$ satisfies the \L ojasiewicz inequality in Definition \ref{type I singularity}
in $B_{r}$, then $x\mapsto r^{-2}u\left(rx\right)$ satisfies the
same \L ojasiewicz inequality in $B_{1}$.} we may also assume that the \L ojasiewicz inequality in Definition
\ref{type I singularity} holds in $B_{1}^{n-1}\times\left(-1,1\right)$.
Moreover, by Theorem \ref{no type I saddle} and the parabolic rescaling,
we may even assume that that $u\left(0\right)=0$ is the maximum value
of $u$ in $B_{1}^{n-1}\times\left(-1,1\right)$.

As is considered in (\ref{continuity of Hessian: maximum point on each level})
in Theorem \ref{continuity of Hessian}, let us define the function
\begin{equation}
u_{\textrm{max}}\left(z\right)=\max_{\left|y\right|\leq1}\,u\left(y,z\right)=\max_{\left|y\right|<\left|z\right|\tan\phi}\,u\left(y,z\right).\label{u_max}
\end{equation}
Note that the last equality in (\ref{u_max}) comes from the the asymptotically
cylindrical behavior of the level set flow in $B_{1}^{n-1}\times\left(-1,1\right)$
outside the cone $\mathscr{C}_{\phi}=\left\{ \left|y\right|\leq\left|z\right|\tan\phi\right\} $.
\begin{lem}
\label{dashed singular set}There exists a sequence $\left\{ z_{i}\right\} _{i\in\mathbb{N}}$
in $\left(-1,1\right)$ that converges to $0$ with $u_{\textrm{max}}\left(z_{i}\right)<0$
for every $i$. 

Upon passing to a subsequence and changing the orientation of the
$z$-axis if necessary, we may assume that $z_{i}\in\left(0,1\right)$
for every $i$.
\end{lem}

\begin{proof}
Suppose the contrary that there exists $\delta>0$ such that
\[
u_{\textrm{max}}\left(z\right)=0\quad\forall-\delta\leq z\leq\delta.
\]
As $0$ is the maximum value of $u$ near the point $0$, the singular
set $\mathcal{S}$, which by Theorem \ref{regularity} is the set
of critical points of $u$, intersects the hyperplane $z=z_{0}$ for
every $z_{0}\in\left[-\delta,\delta\right]$. Note that the points
of intersection are inside the cone $\mathscr{C}_{\phi}$.

On the other hand, near the $1$-cylindrical point $0$ Theorem \ref{rectifiability}
says that $\mathcal{S}=\mathcal{S}_{1}\cup\mathcal{S}_{0}$; the set
$\mathcal{S}_{1}$ is contained in a Lipschitz curve $y=\psi\left(z\right)$
and the set $\mathcal{S}_{0}$ is countable.\footnote{In view of Section \ref{round points}, every round point is isolated
and hence the set $\mathcal{S}_{0}$ is discrete.} Consequently, the set $\mathcal{S}_{1}$ intersects $z=z_{0}$ (inside
the cone $\mathscr{C}_{\phi}$) for all but countably many $z_{0}\in\left[-\delta,\delta\right]$.
It then follows from the upper semicontinuity of Gaussian density
(cf. \cite{E}) that $\mathcal{S}_{1}$ actually fills the Lipschitz
curve $y=\psi\left(z\right)$. By Remark \ref{Reifenberg}, $\mathcal{S}_{1}$
is indeed a $C^{1}$ embedded curve near the a $1$-cylindrical point
$0$; that is, $0$ is a regular singular point by Theorem \ref{continuity of Hessian}.
Thus we get a contradiction. 
\end{proof}
Fix $z_{*}\in\left(0,\frac{1}{2}\right]$ such that $t_{*}=u_{\textrm{max}}\left(z_{*}\right)\in\left[\frac{-1}{8\left(n-2\right)},0\right)$.\footnote{The reason of choosing $t_{*}\geq\frac{-1}{8\left(n-2\right)}$ is
to make sure that the asymptotically cylindrical part of $u\geq t_{*}$
outside the cone $\mathscr{C}_{\phi}$ is roughtly contained in $\bar{B}_{\nicefrac{1}{2}}^{n-1}\times\left(-1,1\right)$
in view of (\ref{asymptotic arrival time outside cone}).} From now on let us turn our attention to the region $B_{1}^{n-1}\times\left[0,z_{*}\right]$
and define the truncated cone 
\[
\mathscr{C}_{\phi}^{*}=\left\{ \left|y\right|\leq z\,\tan\phi\right\} \cap\left\{ 0\leq z\leq z_{*}\right\} .
\]
Note that the contour map of $u$ in $\left(B_{1}^{n-1}\times\left[0,z_{*}\right]\right)\setminus\mathscr{C}_{\phi}^{*}$
is clear from the asymptotically cylindrical behavior of the level
set flow; that is,\footnote{See (\ref{cylindrical arrival time}) with $k=1.$}
\begin{equation}
u\left(y,z\right)\approx\frac{-1}{2\left(n-2\right)}\left|y\right|^{2}\quad\,\,\forall\,z\,\tan\phi\leq\left|y\right|\leq1,\,\,0\leq z\leq z_{*}.\label{asymptotic arrival time outside cone}
\end{equation}
In particular, on the lateral boundary of the cone $\mathscr{C}_{\phi}^{*}$
we have the approximation for the following parabolic-scaling-invariant
quantity\footnote{Given the assumption that $u$ is $\left(\phi,\epsilon\right)$-asymptotically
cylindrical in $B_{1}^{n-1}\times\left(-1,1\right)$, by Theorem \ref{asymptotically cylindrical},
Definition \ref{cylindrical scale}, and Remark \ref{Reifenberg},
the error of the approximation in (\ref{arrival time on boundary of cone})
can be made uniformly small, say no larger than $\frac{\tan^{2}\phi}{4\left(n-2\right)}$,
on the whole lateral boundary (except on the point $0$) by choosing
$\epsilon$ tiny. Moreover, we have 
\[
\lim_{z\rightarrow0}\frac{-u\left(z\omega\,\tan\phi,z\right)}{z^{2}}=\frac{\tan^{2}\phi}{2\left(n-2\right)}.
\]
} 
\begin{equation}
\frac{-u\left(z\omega\,\tan\phi,z\right)}{z^{2}}\approx\frac{\tan^{2}\phi}{2\left(n-2\right)}\label{arrival time on boundary of cone}
\end{equation}
for $0<z\leq z_{*}$ and $\omega\in S^{n-2}$.
\begin{rem}
\label{structure of level set prior to singular time}By Remark \ref{regular value}
every $t\in\left(t_{*},0\right)$ is a regular value of $u$ in $B_{1}^{n-1}\times\left[0,z_{*}\right]$.
Specifically, the set $u=t$ in $B_{1}^{n-1}\times\left[0,z_{*}\right]$
is a non-empty,\footnote{This follows from $u\leq u_{\textrm{max}}\left(z_{*}\right)$ on $B_{1}^{n-1}\times\left\{ z_{*}\right\} $
and (\ref{asymptotic arrival time outside cone}).} compact, smoothly embedded, and strictly mean-convex hypersurface
with boundary on the hyperplane $z=0$;\footnote{Since $u\leq t_{*}$ on $B_{1}^{n-1}\times\left\{ z_{*}\right\} $,
the hypersurface $u=t$ in $B_{1}^{n-1}\times\left[0,z_{*}\right]$
is away from the hyperplane $z=z_{*}$. } moreover, it is asymptotic to a cylinder outside the cone $\mathscr{C}_{\phi}^{*}$.
Let $\hat{\Sigma}_{t}$ be the path component of $u=t$ in $B_{1}^{n-1}\times\left[0,z_{*}\right]$
that includes the asymptotically cylindrical part of $u=t$ outside
the cone $\mathscr{C}_{\phi}^{*}$. It is not hard to see that $\hat{\Sigma}_{t}$
is a compact hypersurface in $B_{1}^{n-1}\times\left[0,z_{*}\right)$
with boundary 
\begin{equation}
\partial\hat{\Sigma}_{t}=\left\{ u=t\right\} \cap\left(B_{1}^{n-1}\times\left\{ 0\right\} \right)=\partial\left(\left\{ u=t\right\} \cap\left(B_{1}^{n-1}\times\left[0,z_{*}\right]\right)\right).\label{path component with boudary}
\end{equation}
Every other path component of $u=t$ in $B_{1}^{n-1}\times\left[0,z_{*}\right]$,
if any, is a closed hypersurface in $B_{1}^{n-1}\times\left(0,z_{*}\right)$
in view of (\ref{path component with boudary}). Additionally, since
all the path components are compact and mutually disjoint, they are
away from each other.
\end{rem}

\begin{defn}
\label{path component}For every $t\in\left(t_{*},0\right)$, let
$\hat{\Omega}_{t}$ be the path component of $u\geq t$ in $B_{1}^{n-1}\times\left[0,z_{*}\right]$
that contains the asymptotically cylindrical part of $u\geq t$ outside
the cone $\mathscr{C}_{\phi}^{*}$. Note that $\hat{\Omega}_{t}$
is away from the hyperplane $z=z_{*}$. 
\end{defn}

In the following lemma we investigate the structure of $\hat{\Omega}_{t}$.
\begin{lem}
\label{characterization of path component}For every $t\in\left(t_{*},0\right)$,
$\hat{\Omega}_{t}$ is the region bounded by $\hat{\Sigma}{}_{t}$
(see Remark \ref{structure of level set prior to singular time})
and the hyperplane $z=0$; moreover, $\hat{\Sigma}_{t}\subset\hat{\Omega}_{t}$.
\end{lem}

\begin{proof}
Fix $t\in\left(t_{*},0\right)$. By the path-connectedness of $\hat{\Sigma}_{t}$,
it is clear that $\hat{\Sigma}_{t}\subset\hat{\Omega}_{t}$. Moreover,
as $t$ is a regular value of $u$ in $B_{1}^{n-1}\times\left[0,z_{*}\right]$
(see Remark \ref{regular value}), we have $\hat{\Sigma}_{t}\subset\partial\hat{\Omega}_{t}$.
In what follows we would like to show that 
\[
\partial\hat{\Omega}_{t}\cap\left\{ z>0\right\} \,\subset\hat{\Sigma}_{t}
\]
by contradiction.

Suppose the contrary that $\left(\partial\hat{\Omega}_{t}\setminus\hat{\Sigma}_{t}\right)\cap\left\{ z>0\right\} \neq\emptyset$.
Since
\[
\partial\hat{\Omega}_{t}\cap\left\{ z>0\right\} \,\subset\,\left\{ u=t\right\} \cap\left(B_{1}^{n-1}\times\left(0,z_{*}\right)\right),
\]
it follows from Remark \ref{structure of level set prior to singular time}
that there exists a closed, path-connected, smoothly embedded, and
strictly mean-convex hypersurface $\Gamma_{t}$ that is contained
in $\left\{ u=t\right\} \cap\left(B_{1}^{n-1}\times\left(0,z_{*}\right)\right)$,
is away from $\hat{\Sigma}_{t}$, and satisfies $\Gamma_{t}\cap\partial\hat{\Omega}_{t}\neq\emptyset$.
By Remark \ref{regular value} and the path-connectedness of $\Gamma_{t}$
we then have $\Gamma_{t}\subset\hat{\Omega}_{t}$. 

By the Jordan-Brouwer separation theorem (cf. \cite{GP}), $\Gamma_{t}$
bounds an open connected set $\mathcal{D}_{t}$ in $B_{1}^{n-1}\times\left(0,z_{*}\right)$,\footnote{As $\Gamma_{t}\subset B_{1}^{n-1}\times\left(0,z_{*}\right)$, the
``outside'' of $\mathbb{R}^{n}\setminus\Gamma_{t}$ includes the
complement of $B_{1}^{n-1}\times\left(0,z_{*}\right)$. Therefore,
the ``inside'' should be contained in $B_{1}^{n-1}\times\left(0,z_{*}\right)$.} that is, $\Gamma_{t}=\partial\mathcal{D}_{t}$. In view of Remark
\ref{no local minimum}, we have $\mathcal{D}_{t}\subset\left\{ u\geq t\right\} .$
It follows that the level set flow near $\partial\mathcal{D}_{t}$,
which by Remark \ref{regular value} is a strictly mean-convex mean
curvature flow of closed and path-connected hypersurfaces, is outside
$\mathcal{D}_{t}$ prior to time $t$ and inside $\mathcal{D}_{t}$
past to time $t$. Namely, in a tubular neighborhood of $\partial\mathcal{D}_{t}$,
$u<t$ outside $\mathcal{D}_{t}$ and $u>t$ inside $\mathcal{D}_{t}$.\footnote{Actually, more can be said. Firstly, there exists $\delta>0$ so that
in a tubular neighborhood of $\Gamma_{t}$, the set $\Gamma_{\tau}=\left\{ u=\tau\right\} $,
where $\tau\in\left[t-\delta,t+\delta\right],$ is a strictly mean-convex
hypersurface in $B_{1}^{n-1}\times\left(0,z_{*}\right)$ that is diffeomorphic
to $\Gamma_{t}$; moreover, $\Gamma_{\tau}$ bounds an open connected
set $\mathcal{D}_{\tau}$ contained in $\left\{ u\geq\tau\right\} \cap\left(B_{1}^{n-1}\times\left(0,z_{*}\right)\right)$.
As the flow $\left\{ \Gamma_{\tau}\right\} $ moves inwardly, we have
$\mathcal{D}_{\tau_{1}}\supset\mathcal{D}_{\tau_{2}}$ for $t-\delta\leq\tau_{1}<\tau_{2}<t+\delta$.
Thus, the domain $\mathcal{D}_{t}$ can be written as $\underset{\tau\in\left(t,t+\delta\right]}{\cup}\Gamma_{\tau}\cup\mathcal{D}_{t+\delta}$,
from which we deduce that $\mathcal{D}_{t}\subset\left\{ u>t\right\} $.
Indeed, the domain $\mathcal{D}_{\tau}$ keeps contracting by the
mean curvature vector until it becomes $\mathcal{S}\cap\mathcal{D}_{t}$
when $\tau=0$. }

Fix $p\in\hat{\Sigma}_{t}\cap\left\{ z=0\right\} $. As $\Gamma_{t}$
and $p$ are included in the path-connected set $\hat{\Omega}_{t}$,
there exists a continuous curve 
\[
x:\left[0,1\right]\rightarrow\mathbb{R}^{n}
\]
such that $x\left(0\right)\in\Gamma_{t}$, $x\left(1\right)=p$, and
$x\left(s\right)\in\hat{\Omega}_{t}$ for every $s\in\left(0,1\right)$.
In addition, since $\Gamma_{t}$ and $\hat{\Sigma}_{t}$ are away
from each other (see Remark \ref{structure of level set prior to singular time}),
there is $\delta>0$ so that $\hat{\Sigma}_{t}$ is out of the $\delta$-tubular
neighborhood of $\Gamma_{t}=\partial\mathcal{D}_{t}$ and that in
the $\delta$-tubular neighborhood of $\partial\mathcal{D}_{t}$ we
have $u>t$ inside $\mathcal{D}_{t}$ and $u<t$ outside $\mathcal{D}_{t}$.
Recall that the signed distance function to $\mathcal{\partial D}_{t}$,
i.e.,
\[
d\left(x\right)=\left\{ \begin{array}{c}
\textrm{dist}\left(x,\partial\mathcal{D}_{t}\right),\quad x\in\mathcal{D}_{t}\\
-\textrm{dist}\left(x,\partial\mathcal{D}_{t}\right),\quad x\notin\mathcal{D}_{t}
\end{array}\right.,
\]
is smooth in a tubular neighborhood of $\partial\mathcal{D}_{t}$
and continuous in $\mathbb{R}^{n}$. So $d\left(x\left(s\right)\right)$
is a continuous function on $\left[0,1\right]$ with $d\left(x\left(0\right)\right)=0$
and $\left|d\left(x\left(1\right)\right)\right|>\delta$. In view
of $x\left(1\right)=p\in\left\{ z=0\right\} $ and $\mathcal{D}_{t}\subset B_{1}^{n-1}\times\left(0,z_{*}\right)$,
we get that $d\left(x\left(1\right)\right)<-\delta$. It then follows
from the intermediate value theorem that there exists $\mathring{s}\in\left(0,1\right)$
such that $d\left(x\left(\mathring{s}\right)\right)=-\delta$, yielding
$u\left(x\left(\mathring{s}\right)\right)<t$. However, we have 
\[
x\left(\mathring{s}\right)\in\hat{\Omega}_{t}\subset\left\{ u\geq t\right\} ,
\]
which is a contradiction. 
\end{proof}
\begin{rem}
\label{containment of truncated cone}For each $t\in\left(t_{*},0\right)$,
by the continuity of $u$ at $0$ there exists $r\in\left(0,z_{*}\right)$
so that $u\geq t$ in $B_{r}^{n-1}\times\left[-r,r\right]$. By Definition
\ref{path component} we then have $\mathscr{C}_{\phi}^{*}\cap\left\{ 0\leq z\leq r\right\} \subset\hat{\Omega}_{t}$.
Thus, it is not hard to see that 
\[
\hat{z}_{t}=\sup\left\{ \zeta\in\left(0,z_{*}\right):\mathscr{C}_{\phi}^{*}\cap\left\{ 0\leq z\leq\zeta\right\} \subset\hat{\Omega}_{t}\right\} 
\]
belongs to $\left(0,z_{*}\right)$ (by virtue of the continuity of
$u$) and that 
\[
\mathscr{C}_{\phi}^{*}\cap\left\{ 0\leq z<\hat{z}_{t}\right\} \subset\hat{\Omega}_{t}.
\]
Note that $\hat{z}_{t}\searrow0$ as $t\nearrow0$.
\end{rem}

Now let us fix $t_{0}\in\left(t_{*},0\right)$. The following proposition
is based on Lemma \ref{characterization of path component} and is
crucial to Lemma \ref{gradient at level maximum point }.
\begin{prop}
\label{path-connected super level set} For every $t\in\left[t_{0},0\right)$
we have
\[
\hat{\Omega}_{t}=\left\{ u\geq t\right\} \cap\hat{\Omega}_{t_{0}};
\]
in particular, $\left\{ u\geq t\right\} \cap\hat{\Omega}_{t_{0}}$
is path-connected.
\end{prop}

\begin{proof}
By Definition \ref{path component} we have $\hat{\Omega}_{t}\subset\,\left\{ u\geq t\right\} \cap\hat{\Omega}_{t_{0}}$
for every $t\in\left[t_{0},0\right)$. Suppose the contrary that there
exists $t_{1}\in\left(t_{0},0\right)$ such that 
\[
\hat{\Omega}_{t_{1}}\subsetneq\,\left\{ u\geq t_{1}\right\} \cap\hat{\Omega}_{t_{0}}.
\]
Then the first thing we would like to show is that
\[
\left\{ u=t_{1}\right\} \cap\hat{\Omega}_{t_{0}}\setminus\hat{\Omega}_{t_{1}}\neq\emptyset.
\]
To see that, choose $p\in\left\{ u\geq t_{1}\right\} \cap\hat{\Omega}_{t_{0}}\setminus\hat{\Omega}_{t_{1}}$.
If $u\left(p\right)=t_{1}$, then we are done. So let us assume that
$u\left(p\right)>t_{1}.$ By the path-connectedness of $\hat{\Omega}_{t_{0}}$,
there exists a continuous curve $x:\left[0,1\right]\rightarrow B_{1}^{n-1}\times\left[0,z_{*}\right]$
such that $x\left(0\right)=p$, $x\left(1\right)\in\hat{\Sigma}_{t_{0}}$,
and $x\left(s\right)\in\hat{\Omega}_{t_{0}}$ for every $s\in\left[0,1\right]$.
Since $u\left(x\left(0\right)\right)>t_{1}$ and $u\left(x\left(1\right)\right)=t_{0}$,
we have
\[
\mathring{s}=\sup\left\{ \sigma\in\left[0,1\right]:u\left(x\left(s\right)\right)\geq t_{1}\,\,\,\forall\,s\in\left[0,\sigma\right]\right\} \,\,\in\left(0,1\right)
\]
and $u\left(x\left(\mathring{s}\right)\right)=t_{1}$; in addition,
as $p=x\left(0\right)\notin\hat{\Omega}_{t_{1}}$ and $x\left(s\right)\in\left\{ u\geq t_{1}\right\} $
for every $s\in\left[0,\mathring{s}\right]$, we have $x\left(\mathring{s}\right)\notin\hat{\Omega}_{t_{1}}$.
Thus, we find $q=x\left(\mathring{s}\right)\in\left\{ u=t_{1}\right\} \cap\hat{\Omega}_{t_{0}}\setminus\hat{\Omega}_{t_{1}}$.

Next, by Remark \ref{structure of level set prior to singular time}
and Lemma \ref{characterization of path component}, there exists
a closed, path-connected, smoothly embedded, and strictly mean-convex
hypersurface $\Gamma_{t_{1}}$ contained in $\left\{ u=t_{1}\right\} \cap\left(B_{1}^{n-1}\times\left(0,z_{*}\right)\right)$
such that $\Gamma_{t_{1}}\cap\hat{\Omega}_{t_{0}}\setminus\hat{\Omega}_{t_{1}}\neq\emptyset$.
Due to the path-connectedness of $\Gamma_{t_{1}}$ and Definition
\ref{path component}, it is clear that $\Gamma_{t_{1}}\subset\hat{\Omega}_{t_{0}}\setminus\hat{\Omega}_{t_{1}}$.
By the argument in proving Lemma \ref{characterization of path component},
$\Gamma_{t_{1}}$ bounds an open connected set $\mathcal{D}_{t_{1}}$
contained in $\left\{ u\geq t_{1}\right\} \cap\left(B_{1}^{n-1}\times\left(0,z_{*}\right)\right)$,
i.e., $\Gamma_{t_{1}}=\partial\mathcal{D}_{t_{1}}$. It then follows
from the path-connectedness\footnote{The interior $\mathcal{D}_{t_{1}}$ is open connected and hence path-connected.
In a tubular neighborhood of the smoothly embedded hypersurface $\partial\mathcal{D}_{t_{1}}$
we can find a continuous path joining a given point on $\partial\mathcal{D}_{t_{1}}$
with a nearby point in $\mathcal{D}_{t_{1}}$. Thus, the closure $\overline{\mathcal{D}}_{t_{1}}$
is path-connected.} of $\overline{\mathcal{D}}_{t_{1}}=\mathcal{D}_{t_{1}}\cup\Gamma_{t_{1}}$
that $\overline{\mathcal{D}}_{t_{1}}\,\subset\hat{\Omega}_{t_{0}}\setminus\hat{\Omega}_{t_{1}}$. 

Due to the type I condition (see Definition \ref{type I singularity}),
on the set $\left\{ t_{0}\leq u\leq t_{1}\right\} \cap\left(B_{1}^{n-1}\times\left[0,z_{*}\right]\right)$
we have 
\[
\left|\nabla u\right|\geq\beta^{-1}\sqrt{-t_{1}}=\rho.
\]
Then by the interior\footnote{Considering the asymptotically cylindrical behavior of $u$ outside
the cone $\mathscr{C}_{\phi}$ and the choice of $z_{*}$, the set
$\left\{ t_{0}\leq u\leq t_{1}\right\} \cap\left(B_{1}^{n-1}\times\left[0,z_{*}\right]\right)$
is roughly contained in $\bar{B}_{\frac{1}{2}}^{n-1}\times\left[0,\frac{1}{2}\right]$.} curvature estimate in \cite{HK}, there exists $\delta>0$ (depending
on the dimension $n$, the scaling-invariant constant of noncollapsing,
the entropy of the flow, $\rho$, and $t_{0}-t_{*}$) so that in a
tubular neighborhood of $\Gamma_{t_{1}}$, the set $\Gamma_{t}=\left\{ u=t\right\} $
for $t\in\left[t_{1}-\delta,t_{1}\right]$ forms a strictly mean-convex
mean curvature flow of hypersurfaces in $B_{1}^{n-1}\times\left(-1,1\right)$
that is diffeomorphic to $\Gamma_{t_{1}}$. In view of the asymptotically
cylindrical behavior of $u$ outside the cone $\mathscr{C}_{\phi}$
and the choice of $z_{*}$, it is clear that $\Gamma_{t}\subset\left\{ u=t\right\} $
is away from $\partial B_{1}^{n-1}\times\left(-1,1\right)$ and stays
strictly below the hyperplane $z=z_{*}$ for every $t\in\left[t_{1}-\delta,t_{1}\right]$.
We would like to show that $\Gamma_{t}$ actually stays in $B_{1}^{n-1}\times\left(0,z_{*}\right)$
for $t\in\left[t_{2},t_{1}\right]$, where $t_{2}=\max\left\{ t_{1}-\delta,t_{0}\right\} $.
To this end, let 
\[
\mathring{t}=\inf\left\{ \tau\in\left[t_{2},t_{1}\right]:\Gamma_{t}\textrm{ is contained in }B_{1}^{n-1}\times\left(0,z_{*}\right)\,\,\forall\,t\in\left[\tau,t_{1}\right]\right\} 
\]
As $\Gamma_{t_{1}}$ is contained in $B_{1}^{n-1}\times\left(0,z_{*}\right)$,
it is clear that $\mathring{t}<t_{1}$. Were $\Gamma_{\mathring{t}}$
not contained in $B_{1}^{n-1}\times\left(0,z_{*}\right)$, it would
be that $\Gamma_{\mathring{t}}\subset B_{1}^{n-1}\times\left[0,z_{*}\right)$
with $\Gamma_{\mathring{t}}\cap\left\{ z=0\right\} \neq\emptyset$.
Choose $\mathring{x}\in\Gamma_{\mathring{t}}\cap\left\{ z=0\right\} $.
Note that $\mathring{x}\neq0$ because $u\left(x\right)=\mathring{t}$.
Since 
\[
z\left(\mathring{x}\right)=0=\min_{\Gamma_{\mathring{t}}}z,
\]
we obtain that $\nabla z\left(\mathring{x}\right)=\left(0,1\right)$
is a normal vector of $\Gamma_{\mathring{t}}\subset\left\{ u=\mathring{t}\right\} $
at $\mathring{x}$ and hence it must be parallel to $\nabla u\left(\mathring{x}\right)$;
a contradiction follows immediately from the asymptotically cylindrical
behavior of $u$ on the hyperplane $z=0$. Therefore, $\Gamma_{\mathring{t}}$
is contained in $B_{1}^{n-1}\times\left(0,z_{*}\right)$ as well;
it follows that $\mathring{t}=t_{2}$.

Furthermore, arguing as in the proof of Lemma \ref{characterization of path component},
for every $t\in\left[t_{2},t_{1}\right]$, $\Gamma_{t}$ bounds an
open connected set $\mathcal{D}_{t}$ contained in $\left\{ u\geq t\right\} \cap\left(B_{1}^{n-1}\times\left(0,z_{*}\right)\right)$
with $\mathcal{\overline{D}}_{t'}\supset\mathcal{\overline{D}}_{t}$
whenever $t_{2}\leq t'<t\leq t_{1}$. As $\mathcal{D}_{t_{1}}\,\subset\hat{\Omega}_{t_{0}}$
and that
\[
\mathcal{\overline{D}}_{t_{2}}=\underset{t\in\left[t_{2},t_{1}\right]}{\cup}\Gamma_{t}\cup\mathcal{D}_{t_{1}},
\]
is path-connected, we have $\mathcal{\overline{D}}_{t_{2}}\subset\hat{\Omega}_{t_{0}}$. 

If $t_{2}=t_{0}$, then we have found an open connected set $\mathcal{D}_{t_{0}}=\mathcal{D}_{t_{2}}$
that is contained in $\hat{\Omega}_{t_{0}}$ whose boundary $\Gamma_{t_{0}}$
is a closed, path-connected, and smoothly embedded hypersurface contained
$\left\{ u=t_{0}\right\} \cap\hat{\Omega}_{t_{0}}\cap\left(B_{1}^{n-1}\times\left(0,z_{*}\right)\right)$.
Else, $t_{2}>t_{0}$, then we will continue the aforementioned process.
As the constant $\delta$ is chosen so that it works for every following
step, so we can set $t_{i+1}=t_{i}-\delta$ for $i=2,3,\cdots$ until
$t_{m-1}\in\left(t_{0},t_{0}+\delta\right)$ for some $m\in\mathbb{N}$
and then we set $t_{m}=t_{0}$. In that case we will have 
\[
\hat{\Omega}_{t_{0}}\supset\overline{\mathcal{D}}_{t_{0}}=\mathcal{\overline{D}}_{t_{m}}\supset\cdots\supset\mathcal{\overline{D}}_{t_{2}}\supset\mathcal{\overline{D}}_{t_{1}},
\]
where each $\mathcal{D}_{t_{i}}$ is an open connected set in $\left\{ u\geq t_{i}\right\} \cap\left(B_{1}^{n-1}\times\left(0,z_{*}\right)\right)$
whose boundary $\Gamma_{t_{i}}$ is a closed, path-connected, and
smoothly embedded hypersurface contained $\left\{ u=t_{i}\right\} \cap\left(B_{1}^{n-1}\times\left(0,z_{*}\right)\right)$.
In particular, we obtain an open connected set $\mathcal{D}_{t_{0}}=\mathcal{D}_{t_{m}}$
contained in $\hat{\Omega}_{t_{0}}$ whose boundary $\Gamma_{t_{0}}$
is a path-connected smooth closed hypersurface contained $\left\{ u=t_{0}\right\} \cap\hat{\Omega}_{t_{0}}\cap\left(B_{1}^{n-1}\times\left(0,z_{*}\right)\right)$.

By Remark \ref{regular value} we have $\Gamma_{t_{0}}\subset\partial\hat{\Omega}_{t_{0}}$.
It follows from Lemma \ref{characterization of path component} that
$\Gamma_{t_{0}}\subset\hat{\Sigma}_{t_{0}}$. We claim that $\Gamma_{t_{0}}$
is relatively clopen in $\hat{\Sigma}_{t_{0}}$; therefore, by the
connectedness of $\hat{\Sigma}_{t_{0}}$ we have $\Gamma_{t_{0}}=\hat{\Sigma}_{t_{0}}$,
giving a contradiction since $\Gamma_{t_{0}}\subset B_{1}^{n-1}\times\left(0,z_{*}\right)$
while $\hat{\Sigma}_{t_{0}}\cap\left\{ z=0\right\} \neq\emptyset$.
To prove the claim, note first that $\Gamma_{t_{0}}$ is compact and
hence it must be relatively closed in $\hat{\Sigma}_{t_{0}}$. To
verify the relatively openness, let us fix $x_{0}\in\Gamma_{t_{0}}$.
By Remark \ref{regular value} there exists $r>0$ so that $B_{r}\left(x_{0}\right)\subset B_{1}^{n-1}\times\left(0,z_{*}\right)$
and that $B_{r}\left(x_{0}\right)\setminus\Gamma_{t_{0}}$ is separated
into two open connected sets: one is contained in $\left\{ u>t_{0}\right\} $
and the other is contained in $\left\{ u<t_{0}\right\} $. As a result,
we have 
\[
\hat{\Sigma}_{t_{0}}\cap B_{r}\left(x_{0}\right)\,\subset\,\left\{ u=t_{0}\right\} \cap B_{r}\left(x_{0}\right)\,\subset\,\Gamma_{t_{0}}\cap B_{r}\left(x_{0}\right);
\]
that is, there is a neighborhood of $x_{0}$ in $\hat{\Sigma}_{t_{0}}$
that is contained in $\Gamma_{t_{0}}$.
\end{proof}
We are about to prove Proposition \ref{isolated singularity in the upper-half },
which says that $0$ is the only singular point in $B_{1}^{n-1}\times\left[0,z_{0}\right]$
for some $z_{0}>0$. The proof relies on the following lemma: 
\begin{lem}
\label{gradient at level maximum point }For every $z\in\left(0,\hat{z}_{t_{0}}\right)$
\footnote{See Remark \ref{containment of truncated cone} for the definition
of $\hat{z}_{t_{0}}$.}such that 
\[
u\left(y,z\right)=u_{\textrm{max}}\left(z\right)\in\left[t_{0},0\right),
\]
where $y\in B_{z\tan\phi}^{n-1}$ is any maximum point of $u$ on
level $z$ (see (\ref{u_max})),\footnote{Note that $\nabla u\left(y,z\right)\neq0$ by the type I condition
(see Definition \ref{type I singularity}) and that it must be parallel
to $\left(0,1\right)$ as $\left(y,z\right)$ is a maximum point of
$u$ on level $z$.} there holds 
\[
\nabla u\left(y,z\right)\cdot\left(0,1\right)<0.
\]
\end{lem}

\begin{proof}
Suppose the contrary that there exists $\mathring{z}\in\left(0,\hat{z}_{t_{0}}\right)$
and $y\in B_{\mathring{z}\tan\phi}^{n-1}$ such that $u\left(\mathring{y},\mathring{z}\right)=u_{\textrm{max}}\left(\mathring{z}\right)\in\left[t_{0},0\right)$
and $\nabla u\left(\mathring{y},\mathring{z}\right)=\left(0,\lambda\right)$
for some $\lambda>0$. Choose $\delta\in\left(0,\hat{z}_{t_{0}}-\mathring{z}\right)$
such that 
\[
t_{0}\leq u\left(\mathring{y},\mathring{z}\right)<u\left(\mathring{y},\mathring{z}+\delta\right)<0.
\]
Since $\mathring{z}+\delta<\hat{z}_{t_{0}}$ and $\mathring{y}<\left(\mathring{z}+\delta\right)\tan\phi$,
by Remark \ref{containment of truncated cone} $\left(\mathring{y},\mathring{z}+\delta\right)$
belongs to $\left\{ u\geq u\left(\mathring{y},\mathring{z}+\delta\right)\right\} \cap\hat{\Omega}_{t_{0}}$.
Then it follows from Proposition \ref{path-connected super level set}
that there exists a continuous curve $x:\left[0,1\right]\rightarrow B_{1}^{n-1}\times\left[0,z_{*}\right]$
so that $x\left(0\right)=\left(\mathring{y},\mathring{z}+\delta\right)$,
$x\left(1\right)=0$, and 
\[
x\left(s\right)\in\left\{ u\geq u\left(\mathring{y},\mathring{z}+\delta\right)\right\} \cap\hat{\Omega}_{t_{0}}\quad\forall\,s\in\left[0,1\right].
\]
Since $z\left(x\left(s\right)\right)$ is a continuous function with
$z\left(x\left(0\right)\right)=\mathring{z}+\delta$ and $z\left(x\left(1\right)\right)=0$,
by the intermediate value theorem there exists $\mathring{s}\in\left(0,1\right)$
so that 
\[
z\left(x\left(\mathring{s}\right)\right)=\mathring{z}.
\]
By the choice of the continuous path, we obtain 
\[
u\left(x\left(\mathring{s}\right)\right)\geq u\left(\mathring{y},\mathring{z}+\delta\right)>u\left(\mathring{y},\mathring{z}\right),
\]
contradicting the assumption that $u_{\textrm{max}}\left(\mathring{z}\right)=u\left(\mathring{y},\mathring{z}\right)$.
\end{proof}
\begin{prop}
\label{isolated singularity in the upper-half }There exists $z_{0}\in\left(0,\hat{z}_{t_{0}}\right)$,
where $\hat{z}_{t_{0}}$ is the constant in Lemma \ref{gradient at level maximum point },
so that the function $u_{\max}$ is nonincreasing on $\left[0,z_{0}\right]$
with $u_{\max}\left(z\right)<0$ for every $z\in\left(0,z_{0}\right]$. 

It follows from Remark \ref{characterization of regular point} and
the asymptotically cylindrical behavior of the level set flow on the
hyperplane $z=0$ that the point $0$ is the only singular point in
$B_{1}^{n-1}\times\left[0,z_{0}\right]$.
\end{prop}

\begin{proof}
Firstly we would like to show that $u_{\max}$ is a Lipschitz continuous
on $\left[0,z_{*}\right]$; in particular, it is absolutely continuous.
To see that, given $z$ and $\tilde{z}$ in $\left[0,z_{*}\right]$,
choose $y$ and $\tilde{y}$ in $B_{1}^{n-1}$ so that $u_{\max}\left(z\right)=u\left(y,z\right)$
and $u_{\max}\left(\tilde{z}\right)=u\left(\tilde{y},\tilde{z}\right)$
(see (\ref{u_max})). Then we have 
\[
u_{\max}\left(z\right)-u_{\max}\left(\tilde{z}\right)\leq u\left(y,z\right)-u\left(y,\tilde{z}\right)\leq\left\Vert \nabla u\right\Vert _{L^{\infty}\left(B_{1}^{n-1}\times\left[0,z_{*}\right]\right)}\left|z-\tilde{z}\right|,
\]
\[
u_{\max}\left(z\right)-u_{\max}\left(\tilde{z}\right)\geq u\left(\tilde{y},z\right)-u\left(\tilde{y},\tilde{z}\right)\geq-\left\Vert \nabla u\right\Vert _{L^{\infty}\left(B_{1}^{n-1}\times\left[0,z_{*}\right]\right)}\left|z-\tilde{z}\right|.
\]
Particularly, since $u_{\max}\left(0\right)=0$ and that $0$ is the
maximum value of $u_{\max}$ on $\left[0,z_{*}\right]$, there exists
$z_{0}\in\left(0,\hat{z}_{t_{0}}\right)$ such that 
\begin{equation}
t_{0}\leq u_{\max}\left(z\right)\leq0\label{isolated singularity in the upper-half: z_0}
\end{equation}
for every $z\in\left[0,z_{0}\right]$. 

Note that an absolutely continuous function is differentiable almost
everywhere. Assume that $u_{\max}$ is differentiable at $z\in\left(0,z_{0}\right)$.
Let us choose $y\in B_{z\tan\phi}^{n-1}$ so that $u_{\max}\left(z\right)=u\left(y,z\right)$.
Then either $u_{\max}\left(z\right)=0$, in which case we have $u'_{\max}\left(z\right)=0$
as $0$ is the maximum value of $u_{\max}$ on $\left[0,z_{*}\right]$;
or $u_{\max}\left(z\right)\in\left[t_{0},0\right)$, in which case
Lemma \ref{gradient at level maximum point } gives $\nabla u\left(y,z\right)\cdot\left(0,1\right)<0$,
yielding 
\[
u'_{\max}\left(z\right)=\lim_{h\nearrow0}\frac{u_{\max}\left(z\right)-u_{\max}\left(z-h\right)}{h}\leq\lim_{h\nearrow0}\frac{u\left(y,z\right)-u\left(y,z-h\right)}{h}
\]
\[
=\nabla u\left(y,z\right)\cdot\left(0,1\right)<0.
\]
Thus, $u'_{\max}\leq0$ almost everywhere.

It follows from the fundamental theorem of calculus (for absolutely
continuous functions) that for every $0\leq z<\tilde{z}\leq z_{0}$
we have
\[
u_{\max}\left(\tilde{z}\right)-u_{\max}\left(z\right)=\int_{z}^{z'}u'_{\max}\left(\zeta\right)d\zeta\,\leq0.
\]
Thus, $u_{\max}$ is nonincreasing on $\left[0,z_{0}\right]$. Moreover,
by Lemma \ref{dashed singular set} there exists a sequence $z_{i}\in\left[0,z_{0}\right]$
tending to $0$ with $u_{\max}\left(z_{i}\right)<0$. Therefore, we
have 
\[
u_{\max}\left(z\right)\leq u_{\max}\left(z_{i}\right)<0\quad\forall\,z\in\left[z_{i},z_{0}\right].
\]
As $z_{i}\rightarrow0$, we obtain $u_{\max}\left(z\right)<0$ for
every $z\in\left(0,z_{0}\right]$.
\end{proof}
For every $p\in B_{1}^{n-1}\times\left[0,z_{0}\right]$ \footnote{The positive constant $z_{0}$ is chosen in Proposition \ref{isolated singularity in the upper-half }.}with
$p\neq0$,\footnote{So that $p$ is not a critical point by Proposition \ref{isolated singularity in the upper-half }.}
by the argument in Proposition \ref{saddle point} \footnote{That is, as $\nabla u$ is Lipschitz continuous by Theorem \ref{regularity},
the existence and uniqueness theorem of ODE yields a unique flow line
of the vector field $\nabla u$ passing through the non-stationary
point $p$. The flow line is smooth owing to Theorem \ref{regularity}
and the regularity theory in ODE. After a reparametrization by the
arc length, we then get the desired curve.}there exists a smooth curve $x\left(s\right)=\Psi\left(p,s\right)$
in $B_{1}^{n-1}\times\left[0,z_{0}\right]$ satisfying 
\[
\frac{d}{ds}x=N\left(x\right)=\frac{\nabla u\left(x\right)}{\left|\nabla u\left(x\right)\right|}
\]
with $x\left(0\right)=\Psi\left(p,0\right)=p$. In fact, $s\mapsto\Psi\left(p,s\right)$
is the unique integral curve of the unit normal vector field $N=\frac{\nabla u}{\left|\nabla u\right|}$
that starting at $p$. Let
\[
\mathscr{C}_{\phi}^{0}=\mathscr{C}_{\phi}\cap\left\{ 0\leq z\leq z_{0}\right\} =\left\{ \left|y\right|\leq z\,\tan\phi\right\} \cap\left\{ 0\leq z\leq z_{0}\right\} .
\]
What follows is showing that if the initial point sits in the cone
$\mathscr{C}_{\phi}^{0}$ with its $u$ value at least $u_{\max}\left(z_{0}\right)$,
then the associated integral curve stays in the cone $\mathscr{C}_{\phi}^{0}$
and converges to $0$.
\begin{prop}
\label{arc length of flow line}For every $p\in\mathscr{C}_{\phi}^{0}$
such that 
\[
u\left(p\right)\in\left[u_{\max}\left(z_{0}\right),0\right),
\]
\footnote{Note that by Proposition \ref{isolated singularity in the upper-half }
we have $u_{\max}\left(z_{0}\right)<0$. Also, the condition that
$u\left(p\right)<0$ implies that $p\neq0$.}the integral curve $s\mapsto\Psi\left(p,s\right)$ of $N=\frac{\nabla u}{\left|\nabla u\right|}$
starting at $p$ stays in $\mathscr{C}_{\phi}^{0}$ and tends to $0$
in finite distance. Moreover, the arc length is bounded above by $2\beta\sqrt{-u\left(p\right)}$.\footnote{The constant $\beta$ is from the type I condition, see Definition
\ref{type I singularity}.}
\end{prop}

\begin{proof}
Fix $p\in\mathscr{C}_{\phi}^{0}\cap\left\{ u_{\max}\left(z_{0}\right)\leq u<0\right\} $
and let $x\left(s\right)=\Psi\left(p,s\right)$. Set
\[
\mathring{s}=\left\{ \sigma\geq0:x\left(s\right)\,\textrm{is defined and stays in the cone }\mathscr{C}_{\phi}^{0}\,\,\textrm{for every}\,\,s\in\left[0,\sigma\right]\right\} .
\]
Firstly we would like to show that $\mathring{s}>0$. This is clear
when $p$ is in the interior of the cone $\mathscr{C}_{\phi}^{0}$,
so our discussion will focus on the case where $p$ is located on
the boundary of the cone $\mathscr{C}_{\phi}^{0}$. If $p$ is on
the top of the cone $\mathscr{C}_{\phi}^{0}$, i.e., $z\left(p\right)=z_{0}$,
then by the conditions $u\left(p\right)\geq u_{\max}\left(z_{0}\right)$
and (\ref{u_max}), we infer that $u\left(p\right)=u_{\max}\left(z_{0}\right)$
and $p\in B_{z_{0}\tan\phi}$.\footnote{So $p$ is not on the edge of the top side.}
It follows from Lemma \ref{gradient at level maximum point } \footnote{Note that $z_{0}\in\left(0,\hat{z}_{t_{0}}\right)$ (see Proposition
\ref{isolated singularity in the upper-half }) and $u_{\max}\left(z_{0}\right)=u\left(p\right)\in\left[t_{0},0\right)$
(see (\ref{isolated singularity in the upper-half: z_0})).}that $\nabla u\left(p\right)=\left(0,-\lambda\right)$ for some $\lambda>0$.
So the curve immediately enters the interior of the cone $\mathscr{C}_{\phi}^{0}$
after $s=0$ \footnote{Note that $N=\frac{\nabla u}{\left|\nabla u\right|}$ is continuous
at $p$.}, yielding that $\mathring{s}>0$. The other possibility is that when
$p$ is on the lateral boundary of the cone with $z\left(p\right)\in\left(0,z_{0}\right)$.
In view of the asymptotically cylindrical behavior of $u$ on the
boundary of the cone $\mathscr{C}_{\phi}^{0}$, we have 
\[
N\left(p\right)\cdotp N_{\partial\mathscr{C}_{\phi}^{0}}\left(p\right)\approx\cos\phi,
\]
\footnote{Recall that by our assumption at the beginning of Section \ref{neckpinch singularities},
$u$ is $\left(\phi,\epsilon\right)$-asymptotically cylindrical in
$B_{1}^{n-1}\times\left(-1,1\right)$. By the discussion in Theorem
\ref{asymptotically cylindrical}, Definition \ref{cylindrical scale},
and Remark \ref{Reifenberg}, the error of the approximation can be
made uniformly small, say no larger than $\frac{1}{2}\cos\phi$, on
the whole $\partial\mathscr{C}_{\phi}^{0}$ by choosing $\epsilon$
tiny. Moreover, we have 
\[
N\left(p\right)\cdotp N_{\partial\mathscr{C}_{\phi}^{0}}\left(p\right)\rightarrow\cos\phi
\]
as $p$ tends to $0$ along the lateral boundary $\partial\mathscr{C}_{\phi}^{0}$. }where $N_{\partial\mathscr{C}_{\phi}^{0}}$ is the inward unit normal
vector of $\partial\mathscr{C}_{\phi}^{0}$. Thus, in this case the
curve also enters the interior of the cone $\mathscr{C}_{\phi}^{0}$
after $s=0$ and hence we have $\mathring{s}>0$.

Next, we show that if $\mathring{s}<\infty$, then the curve must
converge to $0$ as $s\nearrow\mathring{s}$. Note first that as both
$x\left(s\right)$ and $\frac{dx}{ds}=N\left(x\right)$ are bounded
on $\left[0,\mathring{s}\right)$, we have\footnote{By the compactness of $\mathscr{C}_{\phi}^{0}$, there exists a convergent
sequence $x\left(s_{i}\right)\rightarrow q\in\mathscr{C}_{\phi}^{0}$,
where $s_{i}\nearrow\mathring{s}$. It follows from
\[
\left|x\left(\tilde{s}\right)-x\left(s\right)\right|\leq\int_{s}^{\tilde{s}}\left|N\left(x\left(\theta\right)\right)\right|d\theta\leq\tilde{s}-s
\]
for every $0\leq s<\tilde{s}<\mathring{s}$ that $x\left(s\right)\rightarrow q$
as $s\nearrow\mathring{s}$. The reason that $q\in\partial\mathscr{C}_{\phi}^{0}$
is clear.} 
\[
x\left(s\right)\rightarrow q\in\partial\mathscr{C}_{\phi}^{0}
\]
as $s\nearrow\mathring{s}$. Suppose that $q\neq0$, than by the argument
in the last paragraph, the integral curve $s\mapsto\Psi\left(q,s\right)$
is defined on $\left[0,\delta\right]$ for some $\delta>0$ and it
moves to the interior of the cone $\mathscr{C}_{\phi}^{0}$ for $s\in\left(0,\delta\right]$$.$
Using the semigroup property of the flow induced by a vector field,
we have
\[
\Psi\left(p,\mathring{s}+s\right)=\Psi\left(q,s\right),
\]
for $s\in\left[0,\delta\right]$, contradicting the choice of $\mathring{s}$.
Thus, $q=0$.

Lastly, let us see why $\mathring{s}$ is finite and how to estimate
it. By the type I condition, on $\left[0,\mathring{s}\right)$ we
have
\[
\frac{d}{ds}\left[u\left(x\right)\right]=\nabla u\left(x\right)\cdot N\left(x\right)=\left|\nabla u\left(x\right)\right|\geq\frac{1}{\beta}\sqrt{-u\left(x\right)},
\]
yielding 
\[
\sqrt{-u\left(p\right)}=\,\sqrt{-u\left(x\left(0\right)\right)}\geq\sqrt{-u\left(x\left(s\right)\right)}+\frac{s}{2\beta}\,\geq\frac{s}{2\beta}
\]
for every $s\in\left[0,\mathring{s}\right)$. Therefore, $\mathring{s}\leq2\beta\sqrt{-u\left(p\right)}$.
\end{proof}
Now we are ready to get a desired contradiction to finish this subsection.
By Remark \ref{containment of truncated cone} and Proposition \ref{isolated singularity in the upper-half },
we can choose $\check{z}\in\left(0,z_{0}\right)$ and $\check{p}\in\partial\mathscr{C}_{\phi}^{0}$
with $z\left(\check{p}\right)=\check{z}$, i.e., $\check{p}=\left(\check{z}\check{\omega}\,\tan\phi,\check{z}\right)$
for some $\check{\omega}\in S^{n-2}$, so that 
\[
\check{t}=u\left(\check{p}\right)\in\left(u_{\max}\left(z_{0}\right),0\right).
\]
By (\ref{arrival time on boundary of cone}) we have 
\[
\frac{-\check{t}}{\check{z}^{2}}=\frac{-u\left(\check{z}\check{\omega}\,\tan\phi,\check{z}\right)}{\check{z}^{2}}\approx\frac{\tan^{2}\phi}{2\left(n-2\right)}\leq\frac{\tan^{2}\phi}{n-2},
\]
namely, 
\begin{equation}
\check{z}\geq\sqrt{n-2}\,\cot\phi\sqrt{-\check{t}}.\label{length of cylinder}
\end{equation}
On the other hand, by Proposition \ref{arc length of flow line} the
integral curve $s\mapsto\Psi\left(\check{p},s\right)$ of $N=\frac{\nabla u}{\left|\nabla u\right|}$
starting at $\check{p}$ tends to $0$ as $s\nearrow\check{s}$, where
$\check{s}\in\left(0,2\beta\sqrt{-u\left(\check{p}\right)}\right]$
is the arc length of the curve. In particular, we obtain
\begin{equation}
\check{z}\leq\left|\check{p}\right|\leq\check{s}\leq2\beta\sqrt{-u\left(\check{p}\right)}=2\beta\sqrt{-\check{t}}.\label{speed under type I condition}
\end{equation}
In view of (\ref{length of cylinder}) and (\ref{speed under type I condition}),
a contradiction follows if
\begin{equation}
\phi<\tan^{-1}\left(\frac{\sqrt{n-2}}{2\beta}\right).\label{choice of phi}
\end{equation}

\bigskip{}
Department of Mathematics, National Taiwan University, Taipei 106,
Taiwan. 

\smallskip{}
\textit{E-mail:} \texttt{shguo@ntu.edu.tw}
\end{document}